\documentclass[a4paper,10pt,reqno, english]{amsart}  
\usepackage[utf8]{inputenc}
\usepackage[T1]{fontenc}
\usepackage{amssymb}
\usepackage{graphicx}
\usepackage{amsmath,amsthm}
\usepackage{amsfonts,amssymb,enumerate}
\usepackage{url,paralist}
\usepackage{mathtools} 
\usepackage[arrow,curve,matrix,tips,2cell]{xy}  
  \SelectTips{eu}{10} \UseTips
  \UseAllTwocells 

\usepackage{tikz}
\usepackage{pdfsync}
\usepackage{hyperref}
\usepackage{color}
\usepackage{enumerate,amssymb}
\usepackage{mathtools}

\theoremstyle{plain}
\newtheorem{theorem}{Theorem}[section]

\newtheorem{lemma}[theorem]{Lemma}

\newtheorem{corollary}[theorem]{Corollary}
\newtheorem{proposition}[theorem]{Proposition}
\newtheorem{enumeration}[theorem]{Enumeration}

\newtheorem*{theorem*}{Theorem}

\newtheorem*{ghrproblem}{The Gr\"unbaum--Hadwiger--Ramos problem}
\newtheorem*{r-conjecture}{The Ramos conjecture}

\newtheorem*{claim*}{Claim}

\theoremstyle{definition}
\newtheorem{definition}[theorem]{Definition}

\newtheorem{example}[theorem]{Example}
\newtheorem{remark}[theorem]{Remark}

\hypersetup{
  colorlinks = false,
  urlcolor = black,
  pdfkeywords = {albert haase, mathematics, mathematik},
  pdfsubject = {},
  pdfpagemode = UseNone
}

\newcommand{\BIGOP}[1]{\mathop{\mathchoice%
{\raise-0.22em\hbox{\huge $#1$}}%
{\raise-0.05em\hbox{\Large $#1$}}{\hbox{\large $#1$}}{#1}}}

\newcommand{\A}{\mathcal{A}}

\newcommand{\R}{\mathbb{R}}

\newcommand{\Z}{\mathbb{Z}}

\DeclareMathOperator{\im}{im}
\newcommand\Wk{\Sym^{\pm}}
\newcommand\Sym{\mathfrak S}

\newcommand\oo{\mathfrak o}

\newcommand{\bb}{\mathfrak h} 
\newcommand{\M}{\mathcal{M}}
\renewcommand{\H}{\mathcal{H}}
\newcommand{\C}{\mathcal{C}}

\newcommand{\relint}{\operatorname{relint}}

\newcounter{commentcounter}

\begin{document}

\title[Hyperplane mass partitions via relative EOT]{Hyperplane mass partitions via relative equivariant obstruction theory}


\author[Blagojevi\'c]{Pavle V. M. Blagojevi\'{c}} 
\thanks{P.B. was supported by the DFG via the Collaborative Research Center TRR~109 ``Discretization in Geometry and Dynamics'',
       and by the grant ON 174008 of Serbian Ministry of Education and Science.}
\address{Math. Institut SANU, Knez Mihailova 36, 11001 Beograd, Serbia\hfill\break%
\mbox{\hspace{4.2mm}}Inst. Math., FU Berlin, Arnimallee 2, 14195 Berlin, Germany}
\email{pavleb@mi.sanu.ac.rs} 
\author[Frick]{Florian Frick}
\thanks{F.F. and A.H. were supported by DFG via the Berlin Mathematical School.}
\address{Dept.\ Math., Cornell University, Ithaca, NY 14853, USA}
\email{ff238@cornell.edu} 
\author[Haase]{Albert Haase} 
\address{Inst. Math., FU Berlin, Arnimallee 2, 14195 Berlin, Germany} 
\email{a.haase@fu-berlin.de}
\author[Ziegler]{G\"unter M. Ziegler} 
\thanks{G.M.Z.\ received funding from the ERC project  
no.~247029-\emph{SDModels} and by DFG via the Research Training Group “Methods for Discrete Structures” and the 
Collaborative Research Center TRR~109 ``Discretization in Geometry and Dynamics''}  
\address{Inst. Math., FU Berlin, Arnimallee 2, 14195 Berlin, Germany} 
\email{ziegler@math.fu-berlin.de}


\begin{abstract}
The Grünbaum--Hadwiger--Ramos hyperplane mass partition problem was 
introduced by Grünbaum (1960) in a special case and in general form by Ramos (1996).
It asks for the “admissible”
triples $(d,j,k)$ such that for any $j$ masses in $\R^d$ there are 
$k$ hyperplanes that cut each of the masses into $2^k$ equal parts.
Ramos' conjecture is that the Avis--Ramos necessary lower bound condition
$dk\ge j(2^k-1)$ is also sufficient.

We develop a “join scheme” for this problem, such that non-existence of an
${\Sym_k^\pm}$-equivariant map between spheres
$(S^d)^{*k} \rightarrow S(W_k\oplus U_k^{\oplus j})$
that extends a test map on the subspace of $(S^d)^{*k}$
where the hyperoctahedral group $\Sym_k^\pm$ acts non-freely,
implies that $(d,j,k)$ is admissible.

For the sphere $(S^d)^{*k}$ we obtain a very efficient regular cell decomposition,
whose cells get a combinatorial interpretation with respect to measures on a modified moment curve.
This allows us to apply relative equivariant obstruction theory successfully, even in
the case when the difference of dimensions of the spheres $(S^d)^{*k}$ and $S(W_k\oplus U_k^{\oplus j})$ is greater than one.
The evaluation of obstruction classes leads to counting problems for concatenated Gray codes.

Thus we give a rigorous, unified treatment of the previously announced
cases of the Grünbaum--Hadwiger--Ramos problem, as well as a number of new cases
for Ramos' conjecture.
\end{abstract}


\date{September 9, 2015; revised April 27, 2016}

\maketitle

\section{Introduction}
\label{sec : Introduction}

\subsection{Gr\"unbaum--Hadwiger--Ramos hyperplane mass partition problem}
\label{subsec:the_problem}

In 1960, Grünbaum \cite[Sec. 4.(v)]{grunbaum1960partitions}
asked whether for any convex body in $\R^k$ there are $k$ affine hyperplanes that divide it into $2^k$ parts of equal volume:
This is now known to be true for $k\le 3$, due to Hadwiger \cite{hadwiger1966} in 1966, and remains open and challenging for $k=4$.
(A weak partition result for $k=4$ was given in 2009 by Dimitrijevi\'c-Blagojevi\'c \cite{dimitrijevic2009}.)
For $k>4$ it is false, as shown by Avis~\cite{avis1984} in 1984 by considering a measure on a moment curve. 
In 1996, Ramos \cite{ramos1996equipartition} proposed the following generalization of Gr\"unbaum's problem.

\begin{ghrproblem}
\label{GruenbaumProblem}
Determine the minimal dimension $d=\Delta (j,k)$ such that for every collection of $j$ masses
$\mathcal{M}$ on $\R^d$ there exists an arrangement of $k$ affine hyperplanes $\mathcal{H}$ in $\R^d$ that equiparts~$\mathcal{M}$.
\end{ghrproblem}

\noindent
The Ham Sandwich theorem, conjectured by Steinhaus and proved by Banach, states that~$\Delta(d,1)=d$.  
The Gr\"{u}nbaum--Hadwiger--Ramos hyperplane mass partition problem was studied by many authors.  
It has been an excellent testing ground for different equivariant topology methods; see to our recent survey in~\cite{bfhz_2015}.

The first general result about the function $\Delta (j,k)$ was obtained by Ramos~\cite{ramos1996equipartition},
by generalizing Avis' observation:
The lower bound 
\[
\Delta(j,k)\ge\tfrac{2^k-1}{k}j
\]
follows from considering $k$ measures with disjoint connected supports concentrated along a moment curve in~$\R^d$.
Ramos also conjectured that this lower bound is tight.

\begin{r-conjecture}  
$\Delta(j,k)=\lceil \tfrac{2^k-1}{k}j\rceil$ for every $j\geq 1$ and $k\geq 1$.
\end{r-conjecture}

All available evidence up to now supports this, though it has been established rigorously only in special cases.

\subsection{Product scheme and join scheme}
\label{subsec:product_and_join_scheme}

It seems natural to use $Y_{d,k}:=(S^d)^k$ as a configuration space for any $k$ oriented affine hyperplanes/halfspaces in~$\R^d$, which leads to the following \emph{product scheme}:
If there is no equivariant map
\[
   (S^d)^k \longrightarrow_{\Sym_k^\pm} S(U_k^{\oplus j})
\]
from the configuration space to the unit sphere in the space $U_k^{\oplus j}$ of values on the orthants of $\R^k$ that sum to $0$, which is equivariant with respect to the hyperoctahedral (signed permutation) group $\Wk_k$, then there is no counter-example for the given parameters, so $\Delta(j,k)\le d$.

However, our critical review \cite{bfhz_2015} of the main papers on the Gr\"unbaum--Hadwiger--Ramos problem since 1998 has shown that this scheme is very hard to handle:
Except for the 2006 upper bounds by Mani-Levitska, Vre\'cica \& \v{Z}ivaljevi\'c
   \cite{mani2006}, derived from a Fadell--Husseini index calculation, it has produced very few valid results:
The group action on $(S^d)^k$ is not free, the Fadell--Husseini index is rather large and thus yields weak results, and there is no efficient cell complex model at hand.

In this paper, we provide a new approach, which proves to be remarkably clean and efficient. 
For this, we use a \emph{join scheme}, as introduced by Blagojevi\'c and Ziegler \cite{b_l_z_2012}, which takes the form
\[
   F:\quad (S^d)^{*k} \longrightarrow_{\Sym_k^\pm} S(W_k\oplus U_k^{\oplus j}).
\]
Here the domain $(S^d)^{*k}\subseteq\R^{(d+1)\times k}$ is a sphere of dimension $dk+k-1$, given by
\[
X_{d,k}:=\{(\lambda_1x_1,\dots,\lambda_kx_k): x_1,\dots,x_k\in S^d,\ \lambda_1,\dots,\lambda_k\ge0,\ 
\lambda_1+\dots+\lambda_k=1\},
\]
where we write $\lambda_1x_1+\dots+\lambda_kx_k$ as a short-hand for $(\lambda_1x_1,\dots,\lambda_kx_k)$.
The co-domain is a sphere of dimension $j(2^k-1)+k-2$.
Both domain and co-domain are equipped with canonical $\Wk_k$-actions.
We observe that the map restricted to the 
points with non-trivial stabilizer (the “non-free part”)  
\[
   F':\quad X_{d,k}^{>1}\subset(S^d)^{*k}\longrightarrow_{\Sym_k^\pm}S(W_k\oplus U_k^{\oplus j})
\] 
is the same up to homotopy for all test maps.   
If for any parameters $(j,k,d)$ an equivariant extension $F$ of $F'$ 
does not exist, we get that $\Delta(j,k)\le d$.

To decide the existence of this map, or at least obtain necessary criteria, we employ relative equivariant obstruction theory, as explained by tom Dieck \cite[Sect.~II.3]{tom1987transformation}.
This turns out to work beautifully, and have a few remarkable aspects:
\begin{compactitem}[ $\bullet$ ]
\item The Fox--Neuwirth \cite{FoxNeuwirth62}/Bj\"orner--Ziegler \cite{Z19}  combinatorial stratification method yields a simple and efficient cone stratification for 
the space $\R^{(d+1)\times k}$, which is equivariant with respect to the action of $\Wk_k$ on the columns, 
and which respects the arrangement of $k^2$ subspaces of codimension $d$ given by columns of a matrix $(x_1,\dots,x_d)$ being equal, opposite, or zero.
    
\item This yields a small equivariant regular CW complex model for the sphere $(S^d)^{*k}\subseteq\R^{(d+1)\times k}$, for which the the non-free part, given by an arrangement of $k^2$ subspheres of codimension $d+1$, is an invariant  subcomplex.
	The cells $D^S_I(\sigma)$ in the complex are given by combinatorial data.
    
\item To evaluate the obstruction cocycle, we use measures on a non-standard (binomial coefficient) moment curve. 
	For the resulting test map, the relevant cells $D^S_I(\sigma)$ can be interpreted as $k$-tuples of hyperplanes such that some of the hyperplanes have to pass through prescribed points of the moment curve, or equivalently, they have to bisect some extra masses.
\end{compactitem}

\subsection{Statement of the main results}
 
The join scheme reduces the Gr\"unbaum--Hadwiger--Ramos problem 
to a combinatorial counting problem that can be solved by hand or by means of a computer:  
A $k$-bit \emph{Gray code} is a $k\times 2^k $ binary matrix of all column vectors of length~$k$ such that two consecutive vectors 
differ by only one bit. 
Such a $k$-bit code can be interpreted as a Hamiltonian path in the graph of the $k$-cube $[0,1]^k$. 
The \emph{transition count} of a row in a binary matrix~$A$ is the number of bit-changes, \emph{not} counting a bit change from the last to the first entry. By \emph{transition counts} of a matrix~$A$ we refer to the vector of the transition counts of the rows of the matrix~$A$.
Two binary matrices $A$ and $A'$ are {\em equivalent}, if $A$ can be obtained from $A'$ by a sequence of  permutations of rows and/or inversion of bits in rows. 

\begin{definition}
\label{def:Equiparting_Gray_code}
Let $d\geq 1$, $j\geq 1$, $\ell\geq 0$ and $k\geq 1$ be integers such that $dk = (2^k-1)j+ \ell$ with $0\leq \ell\leq d-1$.
A binary matrix $A$ of size $k \times j2^k$ is an {\em $\ell$-equiparting matrix} if 
\begin{compactenum}[\rm ~~~(a)]
\item $A = (A_1, \dots, A_j)$ for Gray codes $A_1,\ldots,A_j$ with the property that the last column of $A_i$ is equal to the first column of $A_{i+1}$ for $1 \le i < j$; and
\item there is one row of the matrix $A$ with the transition count $d-\ell$, while all other rows have transition count $d$. 
\end{compactenum} 	
\end{definition}

\begin{example}
If $d=5$, $j=2$, $\ell=1$ and $k=3$, then a possible $1$-equiparting matrix is
\[ A = (A_1, A_2) =
\begin{pmatrix}
0 & 0 & 1 & 1 & 0 & 0 & 1 & 1 & & 1 & 1 & 1 & 1 & 0 & 0 & 0 & 0 \\
0 & 0 & 0 & 1 & 1 & 1 & 1 & 0 &  & 0 & 1 & 1 & 0 & 0 & 0 & 1 & 1 \\
0 & 1 & 1 & 1 & 1 & 0 & 0 & 0 &  & 0 & 0 & 1 & 1 & 1 & 0 & 0 & 1 \\
\end{pmatrix}.
\]
In this example the first row of~$A$ has transition count~$4$ while the remaining two rows have transition count~$5$.
\end{example} 

\begin{theorem}
\label{th-matrices_corr_equipartitions} 
Let $d\ge 1$, $j\geq 1$, $\ell\geq 0$ and $k\geq 2$ be integers with the property that $dk =(2^k-1)j+ \ell$ and $0\leq \ell\leq d-1$.
The number of non-equivalent $\ell$-equiparting matrices is the number of arrangements of $k$ affine hyperplanes $\H$ that equipart a given collection of $j$ disjoint intervals on a moment curve $\gamma$ in $\R^d$, up to renumbering and orientation change of hyperplanes in $\H$, when it is forced that one of the hyperplanes passes through $\ell$ prescribed points on $\gamma$ that lie to the left of the
$j$ disjoint intervals.	
\end{theorem}

In some situations this yields a solution for the Gr\"unbaum--Hadwiger--Ramos problem.

\begin{theorem}
\label{th-matrices_num_odd_equip_exist}
Let $j\geq 1$ and $k\geq 3$ be integers, with $d:=\lceil \tfrac{2^k-1}{k}j\rceil$ and $\ell:=\lceil \tfrac{2^k-1}{k}j\rceil k-(2^k-1)j=dk-(2^k-1)j$, which implies $0\leq\ell<k\leq d$.
If the number of non-equivalent $\ell$-equiparting matrices of size $k \times j2^k$ is odd, then
\[
\Delta(j,k) = \lceil \tfrac{2^k-1}{k}j\rceil.
\]
\end{theorem}

Theorem~\ref{th-matrices_num_odd_equip_exist} is also true for $k=1$ (and thus $d=j$, $\ell=0$), where it yields the Ham Sandwich theorem: 
In this case an equiparting matrix $A$ is a row vector of length $2d$ and transition count~$d$. 
Thus, each $A_i$ is either $(0,1)$ or $(1,0)$, where $A_i$ uniquely determines $A_{i+1}$. 
Hence, up to inversion of bits $A$ is unique and so~$\Delta(d,1) \leq d$, and consequently $\Delta(d,1)=d$.

While the situation for $k=1$ hyperplane is fully understood, we seem to be far from a complete solution for the case of $k=2$ hyperplanes. 
However, we do obtain the following instances.  

\begin{theorem}
\label{th-main:D(j,2)}
Let $t\geq1$. Then:
\begin{compactenum}[\rm ~~~(i)]
\item $\Delta(2^t-1,2)=3\cdot 2^{t-1}-1$, \label{item:2^t-1}
\item $\Delta(2^t,2)=3\cdot 2^{t-1}$, \label{item:2^t}
\item $\Delta(2^t+1,2)=3\cdot 2^{t-1}+2$. \label{item:2^t+1}
\end{compactenum}
\end{theorem}

The statements~(\ref{item:2^t-1}) and~(\ref{item:2^t+1}) were already known: 
Part~(\ref{item:2^t-1}) is the only case where the lower bound of Ramos and the upper bound of Mani-Levitska, Vre\'cica, and \v Zivaljevi\'c~~\cite[Thm.\,39]{mani2006} coincide.
Part (\ref{item:2^t}) is Hadwiger's result \cite{hadwiger1966} for $t=1$; the general case was previously claimed by Mani-Levitska et al.~\cite[Prop.\,25]{mani2006}. 
However, the proof of the result was incorrect and not recoverable, 
as explained in~\cite[Sec.\,8.1]{bfhz_2015}. 
Here we recover this result by a different method of proof.  
Similarly, statement (\ref{item:2^t+1}) was claimed by \v Zivaljevi\'c~\cite[Thm.\,2.1]{zivaljevic2011equipartitions} 
with a flawed proof; for an explanation of the gap see~\cite[Sec.\,8.2]{bfhz_2015}, where we also gave a 
proof of (\ref{item:2^t+1}) via degrees of equivariant maps ~\cite[Sec.\,5]{bfhz_2015}.
Here we will prove all three cases of Theorem~\ref{th-main:D(j,2)} in a uniform way.

In the case of $k=3$ hyperplanes we prove using Theorem~\ref{th-matrices_num_odd_equip_exist} the following instances of the Ramos conjecture.
\begin{theorem}
\label{th-main:D(j,3)}
~~
\begin{compactenum}[\rm ~~~(i)]
\item $\Delta(2,3)=5$, \label{item:Delta(2,3)}
\item $\Delta(4,3)=10$. \label{item:Delta(4,3)}
\end{compactenum}
\end{theorem}

Statement (\ref{item:Delta(2,3)}) was previously claimed by Ramos~\cite[Sec.\,6.1]{ramos1996equipartition}.
A gap in the method that Ramos developed and used to get this result was explained in~\cite[Sec.\,7]{bfhz_2015}.
It is also claimed by Vre\'cica and \v Zivaljevi\'c in the recent preprint \cite{vrezival_2015} without a proof for the crucial \cite[Prop.\,3]{vrezival_2015}.

The reduction result of Hadwiger and Ramos $\Delta(j,k)\leq\Delta(2j,k-1)$ applied to Theorem~\ref{th-main:D(j,3)} implies the following consequences. 
For details on reduction results see for example~\cite[Sec.\,3.3]{bfhz_2015}.

\begin{corollary}
\label{cor:reduction}
~~
\begin{compactenum}[\rm ~~~(i)]
\item $4\leq \Delta(1,4)\leq 5$,
\item $8\leq \Delta(2,4)\leq 10$.
\end{compactenum}
\end{corollary}

\noindent
Note that $\Delta(1,4)$ is the open case for Gr\"unbaum's original conjecture.

\subsection*{Acknowledgement} We are grateful to the referee for very careful and helpful comments on this paper.

\section{The join configuration space test map scheme and  equivariant obstruction theory}
\label{sec:joint cs/tm scheme}

In this section we develop the join configuration test map scheme that was introduced in~\cite[Sec.\,2.1]{blagojevic2011}.
A sufficient condition for $\Delta(j,k)\le d$ will be phrased in terms of the non-existence of a particular equivariant map between representation spheres. 

\subsection{Arrangements of \emph{k} hyperplanes}
\label{subsec:arrangements}

Let $\hat H=\{x\in\R^d : \langle x,v\rangle =a\}$ be an affine hyperplane determined by a vector $v\in\R^d{\setminus}\{0\}$ and a constant $a\in\R$.
The hyperplane $\hat H$ determines two (closed) halfspaces
\[
\hat H^{0}=\{x\in \R^d :\langle x,v\rangle \geq a\}\qquad\text{and}\qquad 
\hat H^{1}=\{x\in\R^d : \langle x,v\rangle \leq a\}.
\]
Let $\mathcal{H}=(\hat H_1,\ldots,\hat H_k)$ be an arrangement of $k$ affine hyperplanes  in $\R^d$, and let $\alpha=(\alpha_1,\ldots,\alpha_k)\in (\Z/2)^k$.
The {\em orthant} determined by the arrangement $\mathcal{H}$ and $\alpha\in(\Z/2)^k$ is the intersection 
\[
\mathcal{O}_{\alpha}^{\mathcal{H}}=\hat H_1^{\alpha_{1}}\cap\cdots\cap\hat H_k^{\alpha_{k}}.
\]

Let $\mathcal{M}=(\mu_1,\dots,\mu_j)$ be a collection of finite Borel probability measures on $\R^d$ such that the measure of each hyperplane is zero.
Such measures will be called \emph{masses}. 
The assumptions about the measures guarantee that $\mu_i(\hat H^0_s)$ depends continuously 
on $\hat H^0_s$.

An arrangement of affine hyperplanes $\mathcal{H}=(\hat H_1,\ldots,\hat H_k)$ {\em equiparts} the collection of masses $\mathcal{M}=(\mu_1,\ldots,\mu_j)$ if for every element $\alpha\in(\Z/2)^k$ and every $\ell\in\{1,\ldots,j\}$
\[
\mu_{\ell} (\mathcal{O}_{\alpha}^{\mathcal{H}})=\tfrac{1}{2^{k}}.
\]
\subsection{The configuration spaces}
\label{subsec:configuration space}
The space of all oriented affine hyperplanes (or closed affine halfspaces) in $\R^d$ can be parametrized by the sphere $S^d$, where the north pole $e_{d+1}$ and the south pole $-e_{d+1}$ represent hyperplanes at infinity. 
An oriented affine hyperplane in $\R^d$ \emph{at infinity} is the set $\R^d$ or~$\emptyset$, depending on the orientation. 
Indeed, embed $\R^d$ into $\R^{d+1}$ via the map $(\xi_1,\dots, \xi_d)^{t} \longmapsto(1,\xi_1,\ldots,\xi_d)^{t}$. 
Then an oriented affine hyperplane $\hat{H}$ in $\R^d$ defines an oriented affine $(d-1)$-dimensional subspace of $\R^{d+1}$ that extends (uniquely) to an oriented linear hyperplane $H$ in~$\R^{d+1}$. 
The outer unit normal vector that determines the oriented linear hyperplane is a point on the sphere~$S^d$. 

We consider the following configuration spaces that parametrize arrangements of $k$ oriented affine hyperplanes in~$\R^d$:
\begin{compactenum}[\rm ~~~(1)] \label{enum_three_cstm_schemes}
\item \emph{The join configuration space}: $X_{d,k} := (S^d)^{*k} \cong S(\R^{(d+1)\times k})$, 
\item \emph{The product configuration space}: $Y_{d,k} := (S^d)^k$. 
\end{compactenum}
The elements of the join $X_{d,k}$ can be presented as formal convex combinations $\lambda_1 v_1 + \dots + \lambda_k v_k$, where $\lambda_i \geq 0,\,\sum \lambda_i = 1$ and $v_i \in S^d$. 

\subsection{The group actions}
\label{subsec:action on configuration space}

The space of all ordered $k$-tuples of oriented affine hyperplanes in $\R^d$ has natural symmetries: Each hyperplane can change orientation and the hyperplanes can be permuted.
Thus, the group $\Wk_k := (\Z/2)^k \rtimes \Sym_k$ encodes the symmetries of both configuration spaces.

The group $\Wk_k$ acts on $X_{d,k}$ as follows. 
Each copy of $\Z/2$ acts antipodally on the appropriate sphere $S^d$ in the join while the symmetric group $\Sym_k$ acts by permuting factors in the join product. 
More precisely, for $((\beta_1, \dots, \beta_k) \rtimes \pi) \in \Wk_k$ and $\lambda_1 v_1 + \dots + \lambda_k v_k \in X_{d,k}$ the action is given by
\begin{multline*}
((\beta_1, \dots, \beta_k) \rtimes \tau)\cdot (\lambda_1 v_1 + \dots + \lambda_k v_k) =\\
\lambda_{\tau^{-1}(1)} (-1)^{\beta_1} v_{\tau^{-1}(1)} + \dots +  \lambda_{\tau^{-1}(k)} (-1)^{\beta_k} v_{\tau^{-1}(k)}.
\end{multline*}
The product space $Y_{d,k}$ is a subspace of the join $X_{d,k}$ via the diagonal embedding $Y_{d,k} \longrightarrow X_{d,k}, (v_1,\ldots, v_k)\longmapsto \frac{1}{k}v_1 + \dots + \frac{1}{k}v_k$. 
The product $Y_{d,k}$ is an invariant subspace of $X_{d,k}$ with respect to the $\Wk_k$-action and consequently inherits the $\Wk_k$-action from $X_{d,k}$.
For $k \ge 2$, the action of $\Wk_k$ is {\em not} free on either $X_{d,k}$ or $Y_{d,k}$.

The sets of points in the configuration spaces $X_{d,k}$ and $Y_{d,k}$ that have non-trivial stabilizer with respect to the action of $\Wk_k$ can be described as follows:
\begin{multline*}
X_{d,k}^{>1}  = \{\lambda_1 v_1 + \dots + \lambda_k v_k : \\
\lambda_1\cdots\lambda_k=0,\text{  or  }\lambda_s=\lambda_r\text{ and }v_s=\pm v_r\text{ for some }1\leq s<r\leq k\},
\end{multline*}
and
\[
Y_{d,k}^{>1}  = \{(v_1,\ldots, v_k) : v_s=\pm v_r\text{ for some }1\leq s<r\leq k\}.
\]

\subsection{Test spaces}
\label{subsec:test space}
Consider the vector space $\R^{(\Z/2)^k}$, where the group element
$((\beta_1, \dots, \beta_k) \rtimes \tau) \in \Wk_k$  acts on a vector $(y_{(\alpha_1,\dots,\alpha_k)})_{(\alpha_1,\dots,\alpha_k) \in (\Z/2)^k}\in \R^{(\Z/2)^k}$ by acting on its indices as
\begin{equation}
\label{eq:action_U_k}
((\beta_1, \dots, \beta_k) \rtimes \tau) \cdot (\alpha_1, \dots, \alpha_k) = (\beta_1 + \alpha_{\tau^{-1}(1)}, \dots, \beta_k + \alpha_{\tau^{-1}(k)}).
\end{equation}
The subspace of $\R^{(\Z/2)^k}$ defined by
\[
U_k=\Big\{ (y_{\alpha})_{\alpha\in (\Z/2)^k} \in \R^{(\Z/2)^k} : \sum_{\alpha\in (\Z/2)^k} y_{\alpha} =0\Big\}
\]
is $\Wk_k$-invariant and therefore an $\Wk_k$-subrepresentation.

Next we consider the vector space $\R^k$ and its subspace
\[
W_k=\Big\{(z_1,\ldots,z_k)\in\R^k : \sum_{i=1}^kz_i=0\Big\}.
\]
The group $\Wk_k$ acts on $\R^k$ by permuting coordinates, i.e., for $((\beta_1, \dots, \beta_k) \rtimes \tau) \in \Wk_k$ and $(z_1,\ldots,z_k)\in\R^k$ we have
\begin{equation}
\label{eq:action_W_k}
((\beta_1, \dots, \beta_k) \rtimes \tau) \cdot (z_1,\ldots,z_k) = (z_{\tau^{-1}(1)},\ldots,z_{\tau^{-1}(k)}).
\end{equation}
In particular, the subgroup $(\Z/2)^k$ of $\Wk_k$ acts trivially on~$\R^k$.
The subspace $W_k\subset\R^k$ is $\Wk_k$-invariant and consequently a $\Wk_k$-subrepresentation.

\subsection{Test maps}
\label{subsec:test maps}

The product test map associated to the collection of $j$ masses $\mathcal{M}=(\mu_1, \dots, \mu_j)$ from the configuration space $Y_{d,k}$ to the {\em test space} $U_{k}^{\oplus j}$ is defined by
\begin{align*}
\phi_{\mathcal{M}} \colon\qquad\qquad Y_{d,k} &\longrightarrow U_{k}^{\oplus j},\\
(v_1, \dots, v_k) &\longmapsto \Big( \big(\mu_i(H^{\alpha_1}_{v_1} \cap \dots \cap H^{\alpha_k}_{v_k}) - \tfrac{1}{2^k} \big)_{(\alpha_1,\dots,\alpha_k)\in (\Z/2)^k} \Big)_{i \in \{1,\dots, j\}}.
\end{align*}

In this paper we mostly work with the join configuration space $X_{d,k}$. 
The corr\-esponding join test map associated to a collection of $j$ masses $\mathcal{M}=(\mu_1, \dots, \mu_j)$ maps the configuration space $X_{d,k}$ into the related {\em test space} $W_k\oplus U_{k}^{\oplus j}$.
It is defined by
\begin{align*}
\psi_{\mathcal{M}} \colon \qquad\qquad\qquad X_{d,k} &\longrightarrow W_k\oplus U_{k}^{\oplus j},\\
\lambda_1 v_1 +\cdots + \lambda_k v_k &\longmapsto (\lambda_1-\tfrac{1}{k},\ldots,\lambda_k-\tfrac{1}{k})\oplus(\lambda_1\cdots\lambda_k) \cdot \phi_{\mathcal{M}} (v_1, \dots, v_k).
\end{align*}

Both maps $\phi_{\mathcal{M}}$ and $\psi_{\mathcal{M}}$ are $\Wk_k$-equivariant with respect to the actions defined in Sections~\ref{subsec:action on configuration space} and~\ref{subsec:test space}.
Let $S(U_{k}^{\oplus j})$ and $S(W_k\oplus U_{k}^{\oplus j})$ denote the unit spheres in the vector spaces $U_{k}^{\oplus j}$ and $W_k\oplus U_{k}^{\oplus j}$, respectively. 
The maps $\phi_{\mathcal M}$ and $\psi_{\mathcal M}$ are called test maps since we have the following criterion, which reduces finding an equipartition to finding zeros of the test map.

\begin{proposition}
\label{prop:CS/TM}
Let $d\geq 1$, $k\geq 1$, and $j\geq 1$ be integers.
\begin{compactenum}[\rm ~~~(i)]
\item \label{item:test-map-zeros}
Let $\mathcal{M}$ be a collection of $j$ masses on $\R^d$, and let 
\[
\phi_{\mathcal{M}} \colon Y_{d,k} \longrightarrow U_{k}^{\oplus j}
\quad\text{and}\quad
\psi_{\mathcal{M}} \colon X_{d,k} \longrightarrow W_k\oplus U_{k}^{\oplus j}
\]
be the $\Wk_k$-equivariant maps defined above.
If \,$0\in \im\phi_{\mathcal{M}}$, or\, $0\in \im\psi_{\mathcal{M}}$, then there is an arrangement of $k$ affine hyperplanes that equiparts $\mathcal{M}$.
\item \label{item:existence-of-maps}
If there is no $\Wk_k$-equivariant map of either type
\[
Y_{d,k} \longrightarrow S(U_{k}^{\oplus j})
\quad\text{or}\quad
X_{d,k} \longrightarrow S(W_k\oplus U_{k}^{\oplus j}),
\]
then  $\Delta(j,k)\leq d$. 
\end{compactenum}
\end{proposition}

It is worth pointing out that $0\in \im\phi_{\mathcal{M}}$ if and only if  $0\in \im\psi_{\mathcal{M}}$, while the existence of an $\Wk_k$-equivariant map $Y_{d,k} \longrightarrow S(U_{k}^{\oplus j})$ implies the existence of a $\Wk_k$-equivariant map $X_{d,k} \longrightarrow S(W_k \oplus U_{k}^{\oplus j})$ but not vice versa.

The homotopy class of the restrictions of the test maps $\phi_{\mathcal{M}}$ and $\psi_{\mathcal{M}}$ on the set of points with non-trivial stabilizer (as maps avoiding the origin) is independent of the choice of the masses~$\mathcal M$, by the following proposition.

\begin{proposition}
Let $\mathcal{M}$ and $\mathcal{M'}$ be collections of $j$ masses in~$\R^d$.
Then 
\begin{compactenum}[\rm ~~~(i)]
\item $0\notin \im \phi_{\mathcal{M}}|_{Y_{d,k}^{>1}}$ and $0\notin \im \psi_{\mathcal{M}}|_{X_{d,k}^{>1}}$, \label{item:avoiding-zero}
\item $\phi_{\mathcal{M}}|_{Y_{d,k}^{>1}}$ and $\phi_{\mathcal{M'}}|_{Y_{d,k}^{>1}}$ are $\Wk_k$-homotopic as maps 
$Y_{d,k}^{>1}\longrightarrow U_k^{\oplus j}{\setminus}\{0\}$, and \label{item:phi-homotopy}
\item $\psi_{\mathcal{M}}|_{X_{d,k}^{>1}}$ and $\psi_{\mathcal{M'}}|_{X_{d,k}^{>1}}$ are $\Wk_k$-homotopic as maps 
$X_{d,k}^{>1}\longrightarrow (W_k\oplus U_k^{\oplus j}){\setminus}\{0\}$. \label{item:psi-homotopy}
\end{compactenum}
\end{proposition}
\begin{proof}
{\rm (\ref{item:avoiding-zero})} If $(v_1,\ldots,v_k)\in Y_{d,k}^{>1}$, then $v_s=\pm v_r$ for some $1\leq s<r\leq k$.
Consequently, the corresponding hyperplanes $H_{v_i}$ and $H_{v_j}$ coincide, possibly with
opposite orientations.  
Thus some of the orthants associated to the collection of hyperplanes $(H_{v_1},\ldots,H_{v_k})$ 
are empty.
Consequently, Proposition \ref{prop:CS/TM} implies that $0\notin \im \phi_{\mathcal{M}}|_{Y_{d,k}^{>1}}$.

In the case where $\lambda_1 v_1 +\cdots + \lambda_k v_k\in X_{d,k}^{>1}$ the additional case $\lambda_s=0$ for some
$1\leq s\leq k$ may occur. If $\lambda_s=0$, then the $s$-th coordinate of $\psi(\lambda_1 v_1 +\cdots + \lambda_k v_k)\in W_k\oplus U_k^{\oplus j}$ is equal to $-\tfrac1k$, and hence
$0\notin \im \psi_{\mathcal{M}}|_{X_{d,k}^{>1}}$.

{\rm (\ref{item:phi-homotopy})} The equivariant homotopy between $\phi_{\mathcal{M}}|_{Y_{d,k}^{>1}}$ and $\phi_{\mathcal{M'}}|_{Y_{d,k}^{>1}}$ 
is just the linear homotopy in~$U_k^{\oplus j}$.
For this the linear homotopy should not have zeros, compare~\cite[proof of Cor.\,5.4]{bfhz_2015}.
It suffices to prove that for each point $(v_1,\ldots,v_k)\in Y_{d,k}^{>1}$, the points $\phi_{\mathcal{M}}(v_1,\ldots,v_k)$ and $\phi_{\mathcal{M'}}(v_1,\ldots,v_k)$ belong to some affine subspace of the test space that is not linear.

First, observe that $\R^{(\Z/2)^k}$, considered as a real $(\Z/2)^k$ representation, is the real regular representation of  $(\Z/2)^k$  and therefore it decomposes into the direct sum of all real irreducible representations.
For this we use the fact that all real irreducible representations of $(\Z/2)^k$ are $1$-dimensional.
The subspace $U_k$ seen as a real $(\Z/2)^k$ subrepresentation of $(\Z/2)^k$ decomposes as follows:
\begin{equation}
\label{eq 12}
U_k\cong\bigoplus_{\alpha\in(\Z/2)^k{\setminus}\{0\}}V_{\alpha}.	
\end{equation}
Here $V_{\alpha}$ is the $1$-dimensional real representation of $(\Z/2)^k$ determined by $\beta\cdot v=-v$ for $x\in V_{\alpha}$ if and only if $\alpha\cdot\beta:=\sum \alpha_s\beta_s=1\in\Z/2$, for $\beta\in (\Z/2)^k$.
The isomorphism \eqref{eq 12} is given by the direct sum of the projections $\pi_{\alpha}\colon U_k\longrightarrow V_{\alpha}$, $\alpha\in (\Z/2)^k{\setminus}\{0\}$,
\[
(y_{\beta})_{\beta\in  (\Z/2)^k{\setminus}\{0\}}\longmapsto \sum_{\alpha\cdot\beta=1}y_{\beta} - \sum_{\alpha\cdot\beta=0}y_{\beta}.
\]

Now let $v_s=\pm v_r$. 
Consider $\alpha\in(\Z/2)^k$ given by $\alpha_s=1=\alpha_r$ and $\alpha_{\ell}=0$ for $\ell\notin\{s,r\}$, and the corresponding projection 
$\pi_{\alpha}^{\oplus j}\colon U_k^{\oplus j}\longrightarrow V_{\alpha}^{\oplus j}$.
Then 
\[
\pi_{\alpha}^{\oplus j}\circ\phi_{\mathcal{M}}(v_1,\ldots,v_k)=\pi_{\alpha}^{\oplus j}\circ\phi_{\mathcal{M'}}(v_1,\ldots,v_k)=(\pm 1,\ldots,\pm 1).
\]

{\rm (\ref{item:psi-homotopy})} Likewise, the linear homotopy 
between $\psi_{\mathcal{M}}|_{X_{d,k}^{>1}}$ and $\psi_{\mathcal{M'}}|_{X_{d,k}^{>1}}$ is equivariant and avoids zero.
Let $\lambda_1 v_1 + \dots + \lambda_k v_k\in X_{d,k}^{>1}$.
If $\lambda:=\lambda_1\cdots\lambda_k\neq 0$, $\lambda_s=\lambda_r$ and $v_s=\pm v_r$, then 
\[
(\pi_{\alpha}^{\oplus j} \circ \eta \circ \psi_{\mathcal{M}})(\lambda_1 v_1 + \dots + \lambda_k v_k)=(\pi_{\alpha}^{\oplus j} \circ \eta \circ \psi_{\mathcal{M'}})(\lambda_1 v_1 + \dots + \lambda_k v_k)=(\pm\lambda,\ldots,\pm\lambda),
\]
where $\eta\colon W_k\oplus U_k^{\oplus j}\longrightarrow U_k^{\oplus j}$ is the projection.
Finally, in the case when $\lambda_s=0$ for some $1\leq s\leq k$, $\psi_{\mathcal{M}}(\lambda_1 v_1 + \dots + \lambda_k v_k)$ and $\psi_{\mathcal{M'}}(\lambda_1 v_1 + \dots + \lambda_k v_k)$ after projection to the $s$th coordinate of the subrepresentation $W_k$ are equal to $-\tfrac1k$.
\end{proof}

Denote the radial projections by
\[
\rho\colon U_k^{\oplus j}{\setminus}\{0\}\longrightarrow S(U_k^{\oplus j})
\quad\text{  and  }\quad
\nu\colon (W_k\oplus U_k^{\oplus j}){\setminus}\{0\}\longrightarrow S(W_k\oplus U_k^{\oplus j}).
\]
Note that $\rho$ and $\nu$ are $\Wk_k$-equivariant maps. 
Now the criterion stated in Proposition \ref{prop:CS/TM}\,(\ref{item:existence-of-maps}) can be strengthened as follows. 

\begin{theorem}
\label{prop:CS/TM-1}
Let $d\geq 1$, $k\geq 1$ and $j\geq 1$ be integers, and let $\mathcal{M}$ be a collection of $j$ masses in~$\R^d$.
We have the following two criteria:
\begin{compactenum}[\rm ~~~(i)]
	\item If there is no $\Wk_k$-equivariant map 
	\[
		Y_{d,k} \longrightarrow S(U_{k}^{\oplus j})
	\]
	whose restriction to $Y_{d,k}^{>1}$ is $\Wk_k$-homotopic to 
	$\rho\circ\phi_{\mathcal{M}}|_{Y_{d,k}^{>1}}$, then  $\Delta(j,k)\leq d$.
	\item If there is no $\Wk_k$-equivariant map 
	\[
		X_{d,k} \longrightarrow S(W_k\oplus U_{k}^{\oplus j})
	\]
	whose restriction to $X_{d,k}^{>1}$ is $\Wk_k$-homotopic to 
	$\nu\circ\psi_{\mathcal{M}}|_{X_{d,k}^{>1}}$, then  $\Delta(j,k)\leq d$.
\end{compactenum}
\end{theorem}

\subsection{Applying relative equivariant obstruction theory}
\label{subsec:obstruction theory}

In order to 
prove Theorems~\ref{th-matrices_num_odd_equip_exist}, \ref{th-main:D(j,2)}, and~\ref{th-main:D(j,3)} 
via Theorem~\ref{prop:CS/TM-1}(ii), we study the existence of an $\Wk_k$-equivariant map
\begin{equation}
\label{eq:equivariant map}
X_{d,k} \longrightarrow S(W_k\oplus U_{k}^{\oplus j}),
\end{equation}
whose restriction to $X_{d,k}^{>1}$ is $\Wk_k$-homotopic to $\nu\circ\psi_{\mathcal{M}}|_{X_{d,k}^{>1}}$ 
for some fixed collection~$\mathcal{M}$ of $j$ masses in~$\R^d$.
If we prove that such a map cannot exist, Theorems~\ref{th-matrices_num_odd_equip_exist}, \ref{th-main:D(j,2)}, and~\ref{th-main:D(j,3)} follow.

Denote by 
\[
    N_1:=(d+1)k-1
\]  
the dimension of the sphere $X_{d,k}=(S^d)^{*k}$, and by 
\[
    N_2:=(2^k-1)j+k-2
\]
the dimension of the sphere~$S(W_k\oplus U_{k}^{\oplus j})$.

If $N_1\leq N_2$, then 
\[
\dim X_{d,k}=N_1\leq \mathrm{conn}\big( S(W_k\oplus U_{k}^{\oplus j})\big) +1=N_2.
\]
Consequently, all obstructions to the existence of an $\Wk_k$-equivariant map \eqref{eq:equivariant map} vanish and so the map exists.
Here $\mathrm{conn}(\cdot)$ denotes the connectivity of a space.

Therefore, we assume that $N_1>N_2$, 
which is equivalent to the Ramos lower bound  $d\geq \tfrac{2^k-1}{k}j$. 
Furthermore, the following prerequisites for applying equivariant obstruction theory are satisfied:
\begin{compactitem}[ $\bullet$ ]
\item The $N_1$-sphere $X_{d,k}$ can be given the structure of a relative $\Wk_k$-CW complex $X:=(X_{d,k},X_{d,k}^{>1})$ 
with a free $\Wk_k$-action on $X_{d,k}{\setminus}X_{d,k}^{>1}$:
In Section~\ref{sec:cell model} we construct 
an explicit relative $\Wk_k$-CW complex that models $X_{d,k}$.
\item The sphere $S(W_k\oplus U_{k}^{\oplus j})$ is path connected and $N_2$-simple, except in the trivial case of $k=j=1$ when $N_2=0$. 
Indeed, the group $\pi_1(S(W_k\oplus U_{k}^{\oplus j}))$ is abelian for $N_2=1$ 
and trivial for $N_2>1$ and therefore its action on $\pi_{N_2}(S(W_k\oplus U_{k}^{\oplus j}))$ is trivial.
\item The $\Wk_k$-equivariant map $h\colon X_{d,k}^{>1}\longrightarrow S(W_k\oplus U_{k}^{\oplus j})$ 
given by the composition $h :=\nu\circ\psi_{\M}|_{X_{d,k}^{>1}}$, for a fixed collection of $j$ masses $\M$, 
serves as the base map for extension.
\end{compactitem}

\noindent
Since the sphere $S(W_k\oplus U_{k}^{\oplus j})$ is $(N_2-1)$-connected, the map $h$ can be extended to 
a $\Wk_k$-equivariant map from the $N_2$-skeleton $X^{(N_2)}\longrightarrow S(W_k\oplus U_{k}^{\oplus j})$.
A necessary criterion for the existence of the $\Wk_k$-equivariant map~(\ref{eq:equivariant map})
extending $h$
is that the $\Wk_k$-equivariant map $h=\nu\circ\psi_{\M}|_{X_{d,k}^{>1}}$ can be extended to a map from 
the $(N_2+1)$-skeleton~$X^{(N_2+1)}\longrightarrow S(W_k\oplus U_{k}^{\oplus j})$.

Given the above hypotheses, we can apply relative equivariant obstruction theory, as presented by 
tom Dieck~\cite[Sec.\,II.3]{tom1987transformation}, to decide the existence of such an extension. 

If $g$ is an equivariant extension of $h$ to the $N_2$-skeleton $X^{(N_2)}$, then the obstruction 
to extending $g$ to the $(N_2+1)$-skeleton is encoded by the equivariant cocycle 
\[
\oo(g)\in \mathcal{C}_{\Wk_k}^{N_2+1}\big(X_{d,k},X_{d,k}^{>1} \,; \,\pi_{N_2}(S(W_k\oplus U_{k}^{\oplus j}))\big).
\]
The $\Wk_k$-equivariant map $g\colon X^{(N_2)}\longrightarrow S(W_k\oplus U_{k}^{\oplus j})$ 
extends to $X^{(N_2+1)}$ if and only if $\oo(g)=0$. Furthermore, the cohomology class 
\[
[\oo(g)]\in \mathcal{H}_{\Wk_k}^{N_2+1}\big(X_{d,k},X_{d,k}^{>1} \,; \,\pi_{N_2}(S(W_k\oplus U_{k}^{\oplus j}))\big),
\]
vanishes if and only if the restriction $g|_{X^{(N_2-1)}}$ to the $(N_2-1)$-skeleton can be extended to 
the $(N_2+1)$-skeleton $X^{(N_2+1)}$. Any two extensions $g$ and $g'$ of $h$ to the $N_2$-skeleton are 
equivariantly homotopic on the $(N_2-1)$-skeleton and therefore the cohomology classes coincide: $[\oo(g)] = [\oo(g')]$. 
Hence, it suffices to compute the cohomology class 
$[\oo(\nu\circ\psi_{\M}|_{X^{(N_2)}})]$ for a fixed collection of $j$ masses $\M$ with the property that 
$0\notin \im (\psi_{\M}|_{X^{(N_2)}})$.

Let $f$ be the attaching map for an $(N_2+1)$-cell $\theta$ and $e$ its corresponding basis element 
in the cellular chain group $C_{N_2+1}(X_{d,k},X_{d,k}^{>1})$.
Then 
\[
\oo(\nu\circ\psi_{\M}|_{X^{(N_2)}})(e) = [\nu\circ\psi_{\M}\circ f|_{\partial\theta}]
\]
is the homotopy class of the map represented by the composition
\[
\xymatrix@1{
 \partial\theta_j\ar[r]^{f|_{\partial\theta}} & X^{(N_2)}\ar[rrr]^{\nu\circ\psi_{\M}|_{X^{(N_2)}}}  & & & S(W_k\oplus U_{k}^{\oplus j}).
}
\]  
 Since $\partial\theta$ and $S(W_k\oplus U_{k}^{\oplus j})$ are spheres of the same dimension $N_2$, the homotopy class 
$[\nu\circ\psi_{\M}\circ f|_{\partial\theta}]$ is determined by the degree of the map 
$\nu\circ\psi_{\M}\circ f|_{\partial\theta}$.
Here we assume that the $\Wk_k$-CW structure on $X_{d,k}$ is endowed with cell orientations, and in addition an orientation on the sphere $S(W_k\oplus U_{k}^{\oplus j})$ is fixed in advance.
Therefore, the degree of the map $\nu\circ\psi_{\M}\circ f|_{\partial\theta}$ is well-defined.

Let $\alpha:=\psi_{\M}\circ f|_{\partial\theta}$. In order to compute the degree of the map $\nu\circ\alpha$ and 
consequently the obstruction cocycle evaluated at~$e$, fix the collection of measures as follows.
Let $\M$ be the collection of masses $(I_1,\ldots,I_j)$ where $I_{r}$ is the mass concentrated 
on the segment $\gamma((t_r^{1},t_r^{2}))$ of the moment curve in~$\R^d$
\[
\gamma(t)=(t,\tbinom{t}{2},\tbinom{t}{3},\ldots, \tbinom{t}{d} )^t,
\]
such that 
\[
\ell< t_1^1<t_1^2<t_2^1<t_2^2<\cdots<t_j^1<t_j^2,
\]
for an integer $\ell$,  $0\leq \ell\leq d-1$.
The intervals $(I_1,\ldots,I_j)$ determined by numbers $t_r^1<t_r^2$ can be chosen in such a way that $0\notin \im (\psi_{\M}|_{X^{(N_2)}})$.
For every concrete situation in Section \ref{sec:proofs} this is verified directly.

Now consider the following commutative diagram:
\[
\xymatrix{
\partial\theta\ar[r]^{f|_{\partial\theta}}\ar[d] & X^{(N_2)}\ar[rrr]^{\psi_{\M}|_{X^{(N_2)}}} \ar[d] & & & W_k\oplus U_{k}^{\oplus j}{\setminus}\{0\} \ar[d]\ar[r]^{\nu} & S(W_k\oplus U_{k}^{\oplus j})\\
\theta\ar[r]^{f}  & X^{(N_2+1)}\ar[rrr]^{\psi_{\M}|_{X^{(N_2+1)}}}  & & & W_k\oplus U_{k}^{\oplus j}
}
\]
where the vertical arrows are inclusions, and the composition of the lower horizontal maps is denoted by $\beta:=\psi_{\M}|_{X^{(N_2+1)}}\circ f$.
Furthermore, let $B_{\varepsilon}(0)$ be a ball with center $0$ in $W_k\oplus U_{k}^{\oplus j}$ of sufficiently small radius~$\varepsilon>0$.
Set $\widetilde{\theta}:=\theta{\setminus}\beta^{-1}(B_{\varepsilon}(0))$.
Since $\dim\theta=\dim W_k\oplus U_{k}^{\oplus j}$ we can assume 
that the set of zeros $\beta^{-1}(0)\subset\relint\theta$ is finite, say of cardinality $r\geq0$.
Again, in every calculation presented in Section \ref{sec:proofs} this assumption is explicitly verified.
The function $\beta$ is a restriction of the test map and therefore {\em the points in $\beta^{-1}(0)$ correspond to 
arrangements of $k$ hyperplanes $\H$ in $\relint\theta$ that equipart $\M$}.
Moreover, the facts that the measures are intervals on a moment curve and that each hyperplane of the arrangement from  $\beta^{-1}(0)$ cuts the moment curve in $d$ distinct points imply that each zero in $\beta^{-1}(0)$ is isolated and transversal. 
The boundary of $\widetilde{\theta}$ consists of the boundary $\partial\theta$ and $r$ disjoint copies of 
$N_2$-spheres $S_1,\ldots,S_r$, one for each zero of $\beta$ on $\theta$.
Consequently, the fundamental class of $\partial\theta$ is equal to the sum of fundamental classes $\sum [S_i]$ 
in $H_{N_1}(\widetilde{\theta};\Z)$.
Here the  fundamental class of $\partial\theta$ is determined by the cell orientation inherited from the $\Wk_k$-CW structure on $X_{d,k}$.
The fundamental classes of $[S_i]$ are determined in such a way that the equality $[\partial\theta]=\sum [S_i]$ holds.
Thus 
\[
\sum (\nu\circ\beta|_{\widetilde\theta})_{*}([S_i]) = (\nu\circ\beta|_{\widetilde\theta})_{*}([\partial\theta])=(\nu\circ \alpha)_{*}([\partial\theta])=\deg(\nu\circ\alpha)\cdot[S(W_k\oplus U_{k}^{\oplus j})].
\]
Recall, we have fixed the orientation on the sphere $S(W_k\oplus U_{k}^{\oplus j})$ and so the fundamental class $[S(W_k\oplus U_{k}^{\oplus j})]$ is also completely determined.
On the other hand,
\[
\sum (\nu\circ\beta|_{S_i})_{*}([S_i])=\Big(\sum\deg(\nu\circ\beta|_{S_i})\Big)\cdot[S(W_k\oplus U_{k}^{\oplus j})].
\]
Hence, $\deg(\nu\circ\alpha) = \sum\deg(\nu\circ\beta|_{S_i})$ where the sum ranges over all arrangements of $k$ hyperplanes $\H$ in $\relint\theta$ that equipart $\M$; consult~\cite[Prop.\,IV.4.5]{outerelo2009}.
In other words,
\begin{equation}
	\label{eq:obstruction_cocycle}
\oo(\nu\circ\psi_{\M}|_{X^{(N_2)}})(e) = [\nu\circ\psi_{\M}\circ f|_{\partial\theta}]
=\deg(\nu\circ\alpha)\cdot\zeta = \sum\deg(\nu\circ\beta|_{S_i})\cdot\zeta,
\end{equation}
where $\zeta\in \pi_{N_2}(S(W_k\oplus U_{k}^{\oplus j}))\cong\Z$ is a generator, and the sum ranges over all 
arrangements of $k$ hyperplanes $\H$ in $\relint\theta$ that equipart~$\M$.

If in addition we assume that all local degrees $\deg(\nu\circ\beta|_{S_i})$ are $\pm 1$ and that the number 
of arrangements of $k$ hyperplanes $\H$ in $\relint\theta$ that equipart $\M$ is odd, then we 
conclude that  $\oo(\nu\circ\psi_{\M}|_{X^{(N_2)}})(e)\neq 0$. It will turn out that in many situations this 
information implies that the cohomology class $[\oo(\nu\circ\psi_{\M})]$ is not zero, and consequently the related 
$\Wk_k$-equivariant map~(\ref{eq:equivariant map}) does not exist, concluding the proof of corresponding Theorems~\ref{th-matrices_num_odd_equip_exist}, \ref{th-main:D(j,2)}, and~\ref{th-main:D(j,3)}.

\section{A regular cell complex model for the join configuration space}
\label{sec:cell model}
In this section, motivated by methods used in \cite{Z19} and \cite{Z131},  we construct a regular $\Wk_k$-CW model for the join configuration space $X_{d,k} = (S^d)^{*k} \cong S(\R^{(d+1)\times k})$ such that $X_{d,k}^{>1}$ is a $\Wk_k$-CW subcomplex.
Consequently, $(X_{d,k},X_{d,k}^{>1})$ has the structure of a relative $\Wk_k$-CW complex.
For simplicity the cell complex we construct is denoted by $X:=(X_{d,k},X_{d,k}^{>1})$ as well.
The cell model is obtained in two steps:
\begin{compactenum}[\rm (1)]
\item the vector space $\R^{(d+1)\times k}$ is decomposed into a union of disjoint relatively open cones (each containing the origin in its closure) on which the $\Wk_k$-action operates linearly permuting the cones, and then
\item the open cells of a regular $\Wk_k$-CW model are obtained as intersections of these relatively open cones with the unit sphere~$S(\R^{(d+1)\times k})$.
\end{compactenum}

\noindent
The explicit relative $\Wk_k$-CW complex we construct here is an essential object needed for the study of the existence of $\Wk_k$-equivariant maps $X_{d,k} \longrightarrow S(W_k\oplus U_{k}^{\oplus j})$ via the relative equivariant obstruction theory of tom Dieck~\cite[Sec.\,II.3]{tom1987transformation}.

\subsection{Stratifications by cones associated to an arrangement}
The first step in the construction of the $\Wk_k$-CW model is an appropriate stratification of the ambient space $\R^{(d+1)\times k}$.
First we introduce the notion of the stratification of a Euclidean space and collect some relevant properties.

\begin{definition} \label{def_stratification_by_cones}
Let $\A$ be an arrangement of linear subspaces in a Euclidean space $E$. 
A \emph{stratification of $E$ (by cones) associated to $\A$} is a finite collection $\C$ of subsets of $E$ that satisfies the following properties:
\begin{compactenum}[\normalfont (i)]
\item $\C$ consists of finitely many non-empty relatively open polyhedral cones in $E$.
\label{def_item_pw_disj_open_pol_cones}
\item $\C$ is a partition of $E$, i.e., $E = \biguplus_{C \in \mathcal{C}} C$.
\label{def_item_R_is_union_of_strata}
\item The closure $\overline C$ of every cone $C \in \C$ is a union of cones in $\C$.
\label{def_item_closure_is_union_of_strata}
\item Every subspace $A \in \A$ is a union of cones in $\C$. 
\label{def_item_subspaces_unions_of_strata}
\end{compactenum}
An element of the family $\C$ is called a \emph{stratum}.
\end{definition}

\begin{example}
\label{ex:flag_stratification}
Let $E$ be a Euclidean space of dimension $d$, 
let $L$ be a linear subspace of codimension $r$, where $1\leq r\leq d$, and let $\A$ be the arrangement $\{L\}$.
Choose a flag that terminates at $L$, i.e., fix a sequence of linear subspaces in $E$
\begin{equation}
\label{eq:flag-1}
E = L^{(0)}\supset L^{(1)}\supset\dots\supset L^{(r)}= L,	
\end{equation}
so that $\dim L^{(i)} =d-i$.
The family $\C$ associated to the flag \eqref{eq:flag-1} consists of $L$ and of the connected components of the successive complements
\[
L^{(0)}{\setminus}L^{(1)},\ L^{(1)}{\setminus}L^{(2)}, \ \ldots \ , L^{(r-1)}{\setminus}L^{(r)}.
\]
A $L^{(i)}$ is a hyperplane in $L^{(i-1)}$, each of the complements $L^{(i-1)}{\setminus}L^{(i)}$ has two connected components. 
This indeed yields a stratification by cones for the arrangement $\A$ in $E$.  
\end{example}

\begin{definition} \label{def_common_refinement}
Let $(\A_1,\A_2, \dots, \A_n)$ be a collection of arrangements of linear subspaces in the Euclidean space $E$
and let $(\C_1,\C_2 \dots, \C_n)$ be the associated collection of stratifications of $E$ by cones.
The \emph{common refinement} of the stratifications is the family
\[
\C := \{ C_1 \cap C_2 \cap \dots \cap C_n  \neq \emptyset : C_i \in \C_i \text{ for all }i \}.
\]
\end{definition}

In order to verify that the common refinement of stratifications is again a stratification, we use the following elementary lemma.  

\begin{lemma}\label{lem_boundary_intersection_convex_sets}
Let $A_1,\dots,A_n$ be relatively open convex sets in $E$ that have non-empty intersection, 
$A_1\cap\dots\cap A_n\neq\emptyset$. 
Then the following relation holds for the closures:
\[
\overline{A_1\cap\dots\cap A_n} = \overline{A_1} \cap\dots\cap \overline{A_n}.
\]
\end{lemma}

\begin{proof}
The inclusion ``$\subseteq$'' follows directly. 
For the opposite inclusion take $x \in \overline{A_1} \cap\dots\cap \overline{A_n}$. 
Choose a point $y \in A_1\cap\dots\cap A_n\neq\emptyset$ and consider the line segment $(x,y] := \{ \lambda x + (1-\lambda) y : 0 \le \lambda < 1 \}$. 
As each $A_i$ is relatively open, the segment $(x,y]$ is contained in each of the $A_i$ and consequently it is contained in $A_1\cap\dots\cap A_n$.
Thus we obtain a sequence in this intersection converging to $x$, 
which implies that $x \in \overline{A_1\cap\dots\cap A_n}$.
\end{proof}

\begin{proposition}
Given stratifications by cones $\C_1,\C_2 \dots, \C_n$ associated to  
linear subspace arrangements $\A_1,\A_2, \dots, \A_n$, their common refinement is a stratification by cones associated to the subspace arrangement $\A := \A_1\cup\dots\cup\A_n$.
\end{proposition}

\begin{proof}
Properties (\ref{def_item_pw_disj_open_pol_cones}) and (\ref{def_item_R_is_union_of_strata}) of Definition~\ref{def_stratification_by_cones} follow immediately from the definition of the common refinement. 
To verify property (\ref{def_item_subspaces_unions_of_strata}), observe that a subspace $A_t\in\A_{t}$ is a union of strata from $\C_{t}$, say $A_t=\bigcup_{s} U_{t,s}$ where $U_{t,s}\in \C_{t}$. 
Hence, taking the union of intersections $C_1 \cap \dots \cap U_{t,s} \cap \dots \cap C_n$ for all $C_i \in \C_i$ where $i\neq t$, and all $U_{t,s}$ gives~$A_{t}$. 
Property (\ref{def_item_closure_is_union_of_strata}) follows from Lemma~\ref{lem_boundary_intersection_convex_sets}.
\end{proof}

\begin{example}\label{ex:flag_stratification_ref}
Let $E$ be a Euclidean space of dimension $d$ and let $\A=\{L_1,\ldots,L_s\}$ be an arrangement of linear subspaces of $E$.
As in Example~\ref{ex:flag_stratification}, for each of the subspaces $L_i$ in the arrangement $\A$ fix a flag $L_i^{(s)}$ and form the corresponding stratifications $\C_1,\ldots,\C_s$.
The common refinement of stratifications  $\C_1,\ldots,\C_s$ is a stratification by cones associated to the subspace arrangement $\A$. 
\end{example}

An arrangement of linear subspaces is \emph{essential} if the intersection of the 
subspaces in the arrangement is~$\{0\}$. 

\begin{proposition}
The intersection of a stratification $\C$ of $E$ by cones associated to an essential linear subspace arrangement with the sphere $S(E)$ gives a regular CW-complex.
\end{proposition}

\begin{proof}
The elements $C \in \C$ are relative open polyhedral cones. As $\{0\}$ is a stratum by itself, none of the strata contains a line through the origin. 
Thus $C\cap S(E)$ is an open cell, whose closure $\overline C\cap S(E)$ is a
finite union of cells of the form $C'\cap S(E)$, so we get a regular CW complex. 
\end{proof}

\subsection{A stratification of $\R^{(d+1)\times k}$}
Now we introduce the stratification of $\R^{(d+1)\times k}$ that will give us a $\Wk_k$-CW model for $X_{d,k}$. 
One version of it, $\C$, arises from the construction in the previous section.
However, we also give combinatorial descriptions of relatively-open convex cones in the 
stratification $\C'$ directly, for which the action of $\Wk_k$ is evident. We then verify
that $\C$ and $\C'$ coincide. 

\subsubsection{Stratification}
Let elements $x \in \R^{(d+1) \times k}$ be written as $x=(x_1,\dots, x_k)$ where $x_i = (x_{t,i})_{t \in [d+1]}$ is the $i$-th column of the matrix $x$. 
Consider the arrangement $\A$  consisting of the following subspaces:
\begin{eqnarray*}
L_r         &:=& 	\{(x_1,\dots, x_k)\in \R^{(d+1) \times k} : x_r = 0\},\qquad 1\le r\le k\\
L_{r,s}^{+} &:=&    \{(x_1,\dots, x_k)\in \R^{(d+1) \times k} : x_r - x_s = 0\},\quad1\le r<s \le k\\
L_{r,s}^{-} &:=&    \{(x_1,\dots, x_k)\in \R^{(d+1) \times k} : x_r + x_s = 0\},\quad1\le r<s \le k.
\end{eqnarray*}
With each subspace we associate a flag:
\begin{compactenum}[\rm (i)]
\item With $L_r=\{ x_r = 0 \}$ we associate  
\begin{multline*}
\qquad\R^{(d+1)\times k} \supset \{ x_{1,r} = 0 \} \supset \{ x_{1,r} = x_{2,r}= 0 \}\supset \dots\supset\\
\{ x_{1,r} = x_{2,r}=\dots= x_{d+1,r}= 0 \},
\end{multline*}
\item With $L_{r,s}^{+}=\{x_r - x_s = 0 \}$ we associate 
\begin{multline*}
\qquad\R^{(d+1)\times k} \supset \{ x_{1,r}-x_{1,s} = 0 \} \supset \{ x_{1,r}-x_{1,s} =  x_{2,r}-x_{2,s}= 0 \}\supset \dots\supset\\
\{x_{1,r}-x_{1,s} =  x_{2,r}-x_{2,s}=\dots= x_{d+1,r}-x_{d+1,s}= 0 \},
\end{multline*}
\item $L_{r,s}^{-}=\{x_r + x_s = 0 \}$ we associate 
\begin{multline*}
\qquad\R^{(d+1)\times k} \supset \{ x_{1,r}+x_{1,s} = 0 \} \supset \{ x_{1,r}+x_{1,s} =  x_{2,r}+x_{2,s}= 0 \}\supset \dots\supset\\
\{x_{1,r}+x_{1,s} =  x_{2,r}+x_{2,s}=\cdots= x_{d+1,r}+x_{d+1,s}= 0 \}.
\end{multline*}
\end{compactenum}
The construction from Example~\ref{ex:flag_stratification} shows how every subspace in $\A$ leads to a stratification by cones of $\R^{(d+1) \times k}$.
The stratifications associated to the subspaces $L_r,L_{r,s}^+,L_{r,s}^-$ are denoted by $\C_r,\C_{r,s}^{+},\C_{r,s}^{-}$, respectively. 
Now, if we apply Example~\ref{ex:flag_stratification_ref} to this concrete situation we obtain the stratification by cones $\C$ of $\R^{(d+1)\times k}$ associated to the subspace arrangement $\A$.
This means that each stratum of $\C$ is a non-empty intersection of strata from the stratifications $\C_r,\C_{r,s}^{+},\C_{r,s}^{-}$ where $1\le r<s \le k$.

\subsubsection{Partition}
Let us fix:
\begin{compactitem}[ $\bullet$ ]
\item a permutation $\sigma:=(\sigma_1,\sigma_2,\ldots,\sigma_k)\equiv (\sigma_1\sigma_2\ldots\sigma_k) \in\Sym_k$, $\sigma\colon t\mapsto \sigma_t$,
\item a collection of signs $S:=(s_1,\ldots,s_k)\in\{+1,-1\}^k$, and
\item integers $I:=(i_1,\ldots,i_k)\in\{1,\ldots,d+2\}^k$.	
\end{compactitem}
Furthermore, define $x_0$ to be the origin in $\R^{(d+1)\times k}$, $\sigma_0=0$ and $s_0=1$.
Define
\[
C_I^S(\sigma)=C_{i_1,\ldots,i_k}^{s_1,\ldots,s_k}(\sigma_1,\sigma_2,\ldots,\sigma_k)\subseteq \R^{(d+1)\times k}
\]
to be the set of all points $(x_1,\ldots,x_k)\in \R^{(d+1)\times k}$, $x_i=(x_{1,i},\ldots,x_{d+1,i})$, such that for each $1\leq t\leq k$,
\begin{compactitem}[ $\bullet$ ]
\item if  $1\leq i_t\leq d+1$, then $s_{t-1}x_{i_t,\sigma_{t-1}}<s_{t}x_{i_t,\sigma_{t}}$ with $s_{t-1}x_{i',\sigma_{t-1}}=s_{t}x_{i',\sigma_{t}}$ for every $i'<i_t$,
\item if $i_t=d+2$, then $s_{i_{t-1}}x_{\sigma_{t-1}}=s_{i_{t}}x_{\sigma_{t}}$.
\end{compactitem}
Any triple $(\sigma|I|S)\in\Sym_k\times \{1,\ldots, d+2\}^k\times \{+1,-1\}^k$ is called a \emph{symbol}.
In the notation of symbols we abbreviate signs $\{+1,-1\}$ by $\{+,-\}$.
The defining set of ``inequalities'' for the stratum $C_I^S(\sigma)$ is briefly denoted by:
\begin{eqnarray*}
C_I^S(\sigma)&=&C_{i_1,\ldots,i_k}^{s_1,\ldots,s_k}(\sigma_1,\sigma_2,\ldots,\sigma_k)\\
&=&\{ (x_1,\ldots,x_k)\in \R^{(d+1)\times k} : 
0<_{i_1}s_1x_{\sigma_1}<_{i_2}s_2x_{\sigma_2}<_{i_3}\cdots<_{i_k}s_kx_{\sigma_k}\},
\end{eqnarray*}
where $y<_iy'$, for $1\leq i\leq d+1$, means that $y$ and $y'$ agree in the first $i-1$ coordinates and at the $i$-th coordinate $y_i~<~y'_i$.
The inequality $y<_{d+2}y'$ denotes that $y=y'$. 
Each set $C_I^S(\sigma)$ is the relative interior of a polyhedral cone in ($\R^{d+1})^k$ of codimension $(i_1-1)+\cdots+(i_k-1)$, i.e.,
\[
\dim C_{i_1,\ldots,i_k}^{s_1,\ldots,s_k}(\sigma_1,\sigma_2,\ldots,\sigma_k)=(d+2)k-(i_1+\cdots+i_k).
\]
Let $\C'$ denote the family of strata $C_I^S(\sigma)$ defined by all symbols, i.e.,
\[
\C'=\{C_I^S(\sigma) : (\sigma|I|S)\in \Sym_k\times \{1,\ldots, d+2\}^k\times \{+1,-1\}^k\}.
\]
Different symbols can define the same set, and
\[
C_I^S(\sigma)\cap C_{I'}^{S'}(\sigma)\neq\emptyset \ \Longleftrightarrow \ C_I^S(\sigma) = C_{I'}^{S'}(\sigma).
\]
In order to verify that the family $\C'$ is a partition of $\R^{(d+1)\times k}$ it remains to prove that it is a covering.

\begin{lemma}
$\bigcup\C'=\R^{(d+1)\times k}$.
\end{lemma}
\begin{proof}
Let $(x_1,\ldots,x_k)\in \R^{(d+1)\times k}$.
First, choose signs $r_1,\ldots,r_k\in\{+1,-1\}$ so that the vectors $r_1x_1,\ldots,r_kx_k$ are greater or equal to $0\in \R^{(d+1)\times k}$ with respect to the lexicographic order, i.e., the first non-zero coordinate of each of the vectors $r_ix_i$ is greater than zero.
The choice of signs is not unique if one of the vectors $x_i$ is zero.
Next, record a permutation $\sigma\in\Sym_k$ such that
\[
0<_{\rm lex}r_{\sigma_1}x_{\sigma_1}<_{\rm lex}r_{\sigma_2}x_{\sigma_2}
<_{\rm lex}\cdots<_{\rm lex}r_{\sigma_k}x_{\sigma_k},
\]
where $<_{\rm lex}$ denotes the lexicographic order.
The permutation $\sigma$ is not unique if $r_ix_i=r_tx_t$ for some $i\neq t$.
Define $s_i:=r_{\sigma_i}$.
Finally, collect coordinates $i_t$ where vectors $s_{t-1}x_{\sigma_{t-1}}$ and $s_{t}x_{\sigma_{t}}$ first differ, or put $i_t=d+2$ if they coincide.
Thus, $(x_1,\ldots,x_k)\in C_{i_1,\ldots,i_k}^{s_1,\ldots,s_k}(\sigma_1,\sigma_2,\ldots,\sigma_k)$.
\end{proof}

\begin{example}
\label{example:1}
Let $d=0$ and $k=2$. 
Then the plane $\R^2$ is decomposed into the following cones.
There are $8$ open cones of dimension $2$:
\[
\begin{array}{lll}
C_{1,1}^{+,+}(12)&=&\{(x_1,x_2)\in\R^2 : 0<x_1<x_2\},\\
C_{1,1}^{-,+}(12)&=&\{(x_1,x_2)\in\R^2 : 0<-x_1<x_2\},\\
C_{1,1}^{+,-}(12)&=&\{(x_1,x_2)\in\R^2 : 0<x_1<-x_2\},\\
C_{1,1}^{-,-}(12)&=&\{(x_1,x_2)\in\R^2 : 0<-x_1<-x_2\},\\
C_{1,1}^{+,+}(21)&=&\{(x_1,x_2)\in\R^2 : 0<x_2<x_1\},\\
C_{1,1}^{-,+}(21)&=&\{(x_1,x_2)\in\R^2 : 0<-x_2<x_1\},\\
C_{1,1}^{+,-}(21)&=&\{(x_1,x_2)\in\R^2 : 0<x_2<-x_1\},\\
C_{1,1}^{-,-}(21)&=&\{(x_1,x_2)\in\R^2 : 0<-x_2<-x_1\}.
\end{array}
\]
Furthermore, there are $8$ cones of dimension $1$:
\[
\begin{array}{lllll}
C_{1,2}^{+,+}(12)&=&C_{1,2}^{+,+}(21)&=&\{(x_1,x_2)\in\R^2 : 0<x_1=x_2\},\\
C_{1,2}^{-,+}(12)&=&C_{1,2}^{+,-}(21)&=&\{(x_1,x_2)\in\R^2 : 0<-x_1=x_2\},\\
C_{1,2}^{+,-}(12)&=&C_{1,2}^{-,+}(21)&=&\{(x_1,x_2)\in\R^2 : 0<x_1=-x_2\},\\
C_{1,2}^{-,-}(12)&=&C_{1,2}^{-,-}(21)&=&\{(x_1,x_2)\in\R^2 : 0<-x_1=-x_2\},\\
C_{2,1}^{+,+}(12)&=&C_{2,1}^{-,+}(12)&=&\{(x_1,x_2)\in\R^2 : 0=x_1<x_2\},\\
C_{2,1}^{+,-}(12)&=&C_{2,1}^{-,-}(12)&=&\{(x_1,x_2)\in\R^2 : 0=x_1<-x_2\},\\
C_{2,1}^{+,+}(21)&=&C_{2,1}^{-,+}(21)&=&\{(x_1,x_2)\in\R^2 : 0=x_2<x_1\},\\
C_{2,1}^{+,-}(21)&=&C_{2,1}^{-,-}(21)&=&\{(x_1,x_2)\in\R^2 : 0=x_2<-x_1\}.
\end{array}
\]
The origin in $\R^2$ is given by $C_{2,2}^{\pm,\pm}(12)=C_{2,2}^{\pm,\pm}(21)$.
The example illustrates a property of our decomposition of~$\R^{(d+1)\times k}$: 
There is a surjection from symbals to cones that is not a bijection, i.e., 
different symbols can define the same cones.
\end{example}
\begin{figure}[ht!]
\centering
\includegraphics[scale=1]{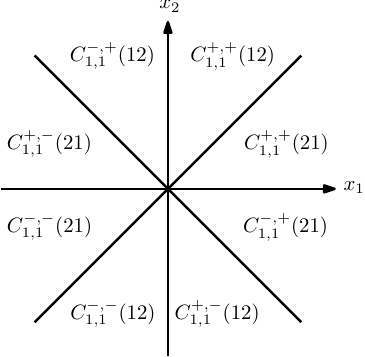}
\label{fig_d=0_k=2_strata}
\caption{\small Illustration of the stratification in Example~\ref{example:1}}
\end{figure}

\begin{example}
Let $d=2$ and $k=4$. 
The stratum associated to the symbol $(2143\,|\,2,3,1,4\,|\,+1,-1,+1,-1)$ can be described explicitly as follows.
\begin{multline*}
\left\{
\left(
\begin{array}{cccc}
x_{1,1}  & x_{1,2}  & x_{1,3} & x_{1,4} \\
x_{2,1}  & x_{2,2}  & x_{2,3} & x_{2,4} \\
x_{3,1}  & x_{3,2}  & x_{3,3} & x_{3,4}
\end{array}
\right) \in (\R^3)^4: \right.\\
\left.
\begin{array}{lllllllll}
0 &=& x_{1,2} &=& -x_{1,1} &<& x_{1,4} &=& -x_{1,3}\\
0 &<& x_{2,2} &=& -x_{2,1} & & x_{2,4} &=& -x_{2,3}\\
  & & x_{3,2} &<& -x_{3,1} & & x_{3,4} &=& -x_{3,3}
\end{array}
\right\}.
\end{multline*}
In particular, 
\[
C_{2,3,1,4}^{+,-,+,-}(2143)=C_{2,3,1,4}^{+,-,-,+}(2134).
\]

\end{example}

\subsubsection{$\C$ and $\C'$ coincide}
We proved that $\C$ is a stratification by cones of $\R^{(d+1)\times k}$, and that $\C'$ is a partition of $\R^{(d+1)\times k}$.
Since both $\C$ and $\C'$ are partitions it suffices to prove that for every symbol $(\sigma|I|S)\in\Sym_k\times \{1,\ldots, d+2\}^k\times \{+1,-1\}^k$ the cone 
$C_I^S(\sigma)\in\C'$ also belongs to $\C$. 

Consider the cone $C_I^S(\sigma)$ in $\C'$.
It is determined by
\begin{eqnarray*}
C_I^S(\sigma)&=&C_{i_1,\ldots,i_k}^{s_1,\ldots,s_k}(\sigma_1,\sigma_2,\ldots,\sigma_k)\\
&=&\{ (x_1,\ldots,x_k)\in \R^{(d+1)\times k} : 
0<_{i_1}s_1x_{\sigma_1}<_{i_2}s_2x_{\sigma_2}<_{i_3}\cdots<_{i_k}s_kx_{\sigma_k}\}.
\end{eqnarray*}
The defining inequalities for $C_I^S(\sigma)$ imply that $(x_1,\ldots,x_k)\in C_I^S(\sigma)$ if and only if 
\begin{compactitem}[ $\bullet$ ]
\item $0<_{\min\{i_1,\ldots,i_a\}}s_ax_a$ for $1\leq a\leq k$, and
\item $s_ax_a<_{\min\{i_{a+1},\ldots,i_b\}}s_bx_b$ for $1\leq a<b\leq k$,
\end{compactitem}
if and only if 
\begin{compactitem}[ $\bullet$ ]
\item $(x_1,\ldots,x_k)$ belongs to the appropriate one of two strata in the complement 
\[
{L_a}^{(\min\{i_1,\ldots,i_a\}-1)}\setminus{L_a}^{(\min\{i_1,\ldots,i_a\}-2)}
\]
of the stratification $\C_a$ depending on the sign $s_a$ where $1\leq a\leq k$, and
\item $(x_1,\ldots,x_k)$ belongs to the appropriate one of two strata in the complement
\[
{L_{a,b}^{s_as_b}}^{(\min\{i_{a+1},\ldots,i_b\}-1)}\setminus
{L_{a,b}^{s_as_b}}_{(\min\{i_{a+1},\ldots,i_b\}-2)}
\]
of the stratification $\C_{a,b}^{s_as_b}$ depending on the sign of the product $s_as_b$ where $1\leq a<b\leq k$. 
The product $s_as_b$, appearing in the ``exponent notation'' of $L_{a,b}^{s_as_b}$, is either ``$+$'' when the product $s_as_b=1$, or ``$-$'' when $s_as_b=-1$.
\end{compactitem}
Here we use the notation of Examples~\ref{ex:flag_stratification} and \ref{ex:flag_stratification_ref}.

Thus we have proved that $C_I^S(\sigma)\in\C$ and consequently $\C=\C'$.  

\subsection{The $\Wk_k$-CW model for $X_{d,k}$}
The action of the group $\Wk_k$ on the space $\R^{(d+1)\times k}$ induces an action on the family of strata $\mathcal{C}$ by as follows:
\begin{eqnarray}
\pi\cdot C^S_I(\sigma) &=& C^S_I(\pi\sigma),\label{eq:action of pi}\\
\varepsilon_t\cdot  C^S_I(\sigma) &=&\varepsilon_t\cdot C_{i_1,\ldots,i_k}^{s_1,\ldots,s_k}(\sigma_1,\sigma_2,\ldots,\sigma_k)\nonumber\\
                                  &=& C_{i_1,\ldots,i_k}^{s_1,\ldots,-s_t,\ldots,s_k}(\sigma_1,\sigma_2,\ldots,\sigma_k),\label{eq:action of epsilon}
\end{eqnarray}
where $\pi\in\Sym_k$ and $\varepsilon_1,\ldots,\varepsilon_k$ are the canonical generators of the subgroup $(\Z/2)^k$ of $\Wk_k$.

The $\Wk_k$-CW complex that models $X_{d,k}=S(\R^{(d+1)\times k})$ is obtained by intersecting each stratum $C^S_I(\sigma)$ with the unit sphere~$S(\R^{(d+1)\times k})$.
Each stratum is a relatively open cone that does not contain a line.
Therefore the intersection 
\[
D^S_I(\sigma)=D_{i_1,\ldots,i_k}^{s_1,\ldots,s_k}(\sigma_1,\sigma_2,\ldots,\sigma_k):=
C_{i_1,\ldots,i_k}^{s_1,\ldots,s_k}(\sigma_1,\sigma_2,\ldots,\sigma_k) \cap S(\R^{(d+1)\times k})
\]
is an open cell of dimension $(d+2)k-(i_1+\cdots+i_k)-1$.
The action of $\Wk_k$ is induced by \eqref{eq:action of pi} and \eqref{eq:action of epsilon}:
\begin{eqnarray}
\pi\cdot D^S_I(\sigma) &=& D^S_I(\pi\sigma),\label{eq:action of pi - 2}\\
\varepsilon_t\cdot  D^S_I(\sigma) &=&\varepsilon_t\cdot D_{i_1,\ldots,i_k}^{s_1,\ldots,s_k}(\sigma_1,\sigma_2,\ldots,\sigma_k)\nonumber\\
                                  &=& D_{i_1,\ldots,i_k}^{s_1,\ldots,-s_t,\ldots,s_k}(\sigma_1,\sigma_2,\ldots,\sigma_k).\label{eq:action of epsilon - 2}
\end{eqnarray}
Thus we have obtained a regular $\Wk_k$-CW model for $X_{d,k}$.
In particular, the action of the group $\Wk_k$ on the space $\R^{(d+1)\times k}$ induces a cellular action on the model.
\begin{theorem}
\label{th : CW-model}
Let $d\geq 1$ and $k\geq 1$ be integers, and $N_1=(d+1)k-1$. 
The family of cells 
\[
\{D^S_I(\sigma) : (\sigma|I|S)\neq (\sigma|d+2,\ldots,d+2|S)\}
\]
forms a finite regular $N_1$-dimensional $\Wk_k$-CW complex $X:=(X_{d,k},X_{d,k}^{>1})$ that models the join configuration space $X_{d,k}=S(\R^{(d+1)\times k})$.
It has 
\begin{compactitem}[ $\bullet$ ]
\item one full $\Wk_k$-orbit in maximal dimension $N_1$, and 
\item $k$ full $\Wk_k$-orbits in dimension $N_1-1$.
\end{compactitem}
The (cellular) $\Wk_k$-action on $X_{d,k}$ is given by \eqref{eq:action of pi - 2} and \eqref{eq:action of epsilon - 2}.
Furthermore the collection of cells 
\[
\{D^S_I(\sigma) : i_s=d+2\text{  for some  }1\leq s\leq k \}
\]
is a $\Wk_k$-CW subcomplex and models $X_{d,k}^{>1}$.
\end{theorem}
\begin{example}
\label{example:boundary-1}
Let $d\ge 1$ and $k\geq 2$ be integers with 
$dk=(2^k-1)j+\ell$, where $0\leq \ell\leq d-1$.
Consider the cell $\theta:=D^{+,+,+,\ldots,+}_{\ell+1,1,1,\ldots,1}(1,2,3,\ldots,k)$ of dimension $N_1-\ell=N_2+1$ in $X_{d,k}$.
It is determined by the following inequalities:
\[
0<_{\ell+1}x_1<_1 x_2<_1 \cdots <_1 x_k.
\] 
For the process of determining the boundary of $\theta$, depending on value of $\ell$, we distinguish the following cases.
\begin{compactenum}[\bf (1)]
\item Let $\ell=0$. 
      Then $\theta:=D^{+,+,+,\ldots,+}_{1,1,1,\ldots,1}(1,2,3,\ldots,k)$.
      The cells of codimension $1$ in the boundary of $\theta$ 
      are obtained by introducing one of the following extra equalities: 
      \[
      x_{1,1}=0\,,\quad x_{1,1}=x_{1,2}\,,\quad \ldots\quad x_{1,k-1}=x_{1,k}.
      \]
      Each of these equalities will give two cells of dimension $N_2$, hence in total $2k$
      cells of codimension $1$, in the boundary of $\theta$.
      \begin{compactenum}[\bf (a)]
      \item The equality $x_{1,1}=0$ induces cells:
      \[
      \qquad\qquad\quad \gamma_1:=D^{+,+,+,\ldots,+}_{2,1,1,\ldots,1}(1,2,3,\ldots,k),\qquad \gamma_2:=D^{-,+,+,\ldots,+}_{2,1,1,\ldots,1}(1,2,3,\ldots,k)
      \]
      that are related, as sets, via $\gamma_2=\varepsilon_1\cdot\gamma_1$.
      Both cells $\gamma_1$ and $\gamma_2$ belong to the linear subspace 
      \[
      V_1=\{(x_1,\ldots,x_k)\in \R^{(d+1)\times k} : x_{1,1}=0\}.  
      \]
      \item The equality $x_{1,r-1}=x_{1,r}$ for $2\leq r\leq k$ gives cells:
      \[
         \gamma_{2r-1}:=D^{+,+,+,\ldots,+}_{1,\ldots,1,2,1,\ldots,1}(1,\ldots,r-1,r,r+1,\ldots,k),
      \]
      \[
         \gamma_{2r}:=D^{+,+,+,\ldots,+}_{1,\ldots,1,2,1,\ldots,1}(1,\ldots,r,r-1,r+1,\ldots,k)
      \]
      satisfying $\gamma_{2r}=\tau_{r-1,r}\cdot\gamma_{2r-1}$. 
      In these cells the index $2$ in the subscript $1,\ldots,1,2,1,\ldots,1$ appears at the position $r$.
	  These cells belong to the linear subspace 
      \[
      V_r=\{(x_1,\ldots,x_k)\in \R^{(d+1)\times k} : x_{1,r-1}=x_{1,r}\}.  
      \]
      \end{compactenum}
      Let $e_{\theta}$ denote a generator in $C_{N_2+1}(X_{d,k},X_{d,k}^{>1})$ 
      that corresponds to the cell~$\theta$. 
      Furthermore let $e_{\gamma_1},\ldots,e_{\gamma_{2k}}$ denote generators in $C_{N_2}(X_{d,k},X_{d,k}^{>1})$ related to the cells $\gamma_1,\ldots,\gamma_{2k}$.\newline
      The boundary of the cell~$\theta$ is contained in the union of the linear subspaces $V_1,\ldots,V_k$.
      Therefore we can orient the cells $\gamma_{2i-1},\gamma_{2i}$ consistently with the orientation of $V_i$, $1\leq i\leq k$, that is given in such a way that
      \[
      \partial e_{\theta}=(e_{\gamma_1}+e_{\gamma_2})+(e_{\gamma_3}+e_{\gamma_4})+\cdots + (e_{\gamma_{2k-1}}+e_{\gamma_{2k}}).
      \]
      Consequently,
      \begin{equation}
      	\label{rel - 1}
      \partial e_{\theta}=(1+(-1)^d\varepsilon_1)\cdot e_{\gamma_1}+\sum_{i=2}^{k}
      (1+(-1)^d\tau_{i-1,i})\cdot e_{\gamma_{2i-1}}.
      \end{equation}
      
\item Let $\ell=1$. 
      Then $\theta:=D^{+,+,+,\ldots,+}_{2,1,1,\ldots,1}(1,2,3,\ldots,k)$. 
      Now the cells in the boundary of $\theta$ are obtained by introducing extra equalities: 
      \[
      x_{2,1}=0\,,\quad (0=)\,x_{1,1}=x_{1,2}\,,\quad \ldots\quad x_{1,k-1}=x_{1,k}.
      \]
      Each of these equalities, except for the second one, will give two cells of dimension $N_2$, 
      which yields $2(k-1)$ cells in total, in the boundary of $\theta$.
      The equality $x_{1,1}=x_{1,2}$ will give additional four cells in the boundary of $\theta$.
      \begin{compactenum}[\bf (a)]
      \item The equality $x_{2,1}=0$ induces cells:
      \[
      \qquad\qquad\quad \gamma_1:=D^{+,+,+,\ldots,+}_{3,1,1,\ldots,1}(1,2,3,\ldots,k),\qquad \gamma_2:=D^{-,+,+,\ldots,+}_{3,1,1,\ldots,1}(1,2,3,\ldots,k)
      \]
      that are related, as sets, via $\gamma_2=\varepsilon_1\cdot\gamma_1$.
      Notice that both cells $\gamma_1$ and $\gamma_2$ belong to the linear subspace 
      \[
      V_1=\{(x_1,\ldots,x_k)\in \R^{(d+1)\times k} : x_{1,1}=x_{2,1}=0\}.  
      \]
      \item The equality $x_{1,1}=x_{1,2}$ yields the cells
      \[
      \qquad\qquad\quad \gamma_3:=D^{+,+,+,\ldots,+}_{2,2,1,\ldots,1}(1,2,3,\ldots,k),\qquad \gamma_{31}:=D^{+,-,+,\ldots,+}_{2,2,1,\ldots,1}(1,2,3,\ldots,k),
      \]
      \[
      \qquad\qquad\quad \gamma_{32}:=D^{+,+,+,\ldots,+}_{2,2,1,\ldots,1}(2,1,3,\ldots,k),\qquad \gamma_{33}:=D^{-,+,+,\ldots,+}_{2,2,1,\ldots,1}(2,1,3,\ldots,k).
      \]
      They satisfy set identities $\gamma_{31}=\varepsilon_2 \cdot\gamma_3$, $\gamma_{32}=\tau_{1,2} \cdot\gamma_3$, and $\gamma_{33}=\varepsilon_1 \tau_{1,2} \cdot\gamma_3$.
      All four cells belong to the linear subspace 
      \[
      V_2=\{(x_1,\ldots,x_k)\in \R^{(d+1)\times k} : 0=x_{1,1}=x_{1,2}\}.  
      \]
      \item The equality $x_{1,r-1}=x_{1,r}$ for $3\leq r\leq k$ gives cells:
      \[
         \gamma_{2r-1}:=D^{+,+,+,\ldots,+}_{2,\ldots,1,2,1,\ldots,1}(1,\ldots,r-1,r,r+1,\ldots,k),
      \]
      \[
         \gamma_{2r}:=D^{+,+,+,\ldots,+}_{2,\ldots,1,2,1,\ldots,1}(1,\ldots,r,r-1,r+1,\ldots,k)
      \]
      satisfying $\gamma_{2r}=\tau_{r-1,r}\cdot\gamma_{2r-1}$. 
      In these cells the second index $2$ in the subscript $2,\ldots,1,2,1,\ldots,1$ appears at the position $r$.
	  These cells belong to the linear subspace 
      \[
      \qquad\qquad V_r=\{(x_1,\ldots,x_k)\in \R^{(d+1)\times k} : x_{1,1}=0,\,x_{1,r-1}=x_{1,r}\}.  
      \]
      \end{compactenum}
      Again $e_{\theta}$ denotes a generator in $C_{N_2+1}(X_{d,k},X_{d,k}^{>1})$ corresponding to $\theta$. 
      Let $e_{\gamma_1},e_{\gamma_2},e_{\gamma_3},e_{\gamma_{31}},e_{\gamma_{32}},e_{\gamma_{33}}, e_{\gamma_{4}}\ldots,e_{\gamma_{2k}}$ denote generators in $C_{N_2}(X_{d,k},X_{d,k}^{>1})$ related to the cells $\gamma_1,\gamma_2,\gamma_3,\gamma_{31},\gamma_{32},\gamma_{33},\ldots,\gamma_{2k}$.\newline
      The boundary of the cell~$\theta$, as before, is contained in the union of the linear subspaces $V_1,\ldots,V_k$.
      Therefore we can orient cells consistently with the orientation of $V_i$, $1\leq i\leq k$, that is given in such a way that
      \[
      \partial e_{\theta}=(e_{\gamma_1}+e_{\gamma_2})+(e_{\gamma_3}+e_{\gamma_{31}}+e_{\gamma_{32}}+e_{\gamma_{33}})+\cdots + (e_{\gamma_{2k-1}}+e_{\gamma_{2k}}).
      \]
      Consequently,
\begin{eqnarray}
\label{rel - 2}
\partial e_{\theta} &=& (1+(-1)^{d-1}\varepsilon_1)\cdot e_{\gamma_1}+\\
       & &(1+(-1)^d\varepsilon_2+(-1)^d\tau_{1,2}+(-1)^{d+d}\varepsilon_1\tau_{1,2})\cdot e_{\gamma_3}+\nonumber \\
      & & \sum_{i=3}^{k}(1+(-1)^d\tau_{i-1,i})\cdot e_{\gamma_{2i-1}}.\nonumber
\end{eqnarray}

\item Let $2\leq \ell \leq d-1$. 
      Then $\theta:=D^{+,+,+,\ldots,+}_{\ell+1,1,1,\ldots,1}(1,2,3,\ldots,k)$.
      The cells in the boundary of $\theta$ are now obtained by introducing following equalities: 
      \[
      x_{\ell+1,1}=0\,,\quad (0=)x_{1,1}=x_{1,2}\,,\quad \ldots\quad x_{1,k-1}=x_{1,k}.
      \]
      Each of them will give two cells of dimension $N_2$ in the boundary of $\theta$, all together $2k$.
      \begin{compactenum}[\bf (a)]
      \item The equality $x_{\ell+1,1}=0$ induces cells:
      \[
      \qquad\qquad\quad \gamma_1:=D^{+,+,+,\ldots,+}_{\ell+2,1,1,\ldots,1}(1,2,3,\ldots,k),\qquad \gamma_2:=D^{-,+,+,\ldots,+}_{\ell+2,1,1,\ldots,1}(1,2,3,\ldots,k)
      \]
      that are related, as sets, via $\gamma_2=\varepsilon_1\cdot\gamma_1$.
      Both cells $\gamma_1$ and $\gamma_2$ belong to the linear subspace 
      \[
      V_1=\{(x_1,\ldots,x_k)\in \R^{(d+1)\times k} : x_{1,1}=\cdots=x_{\ell+1,1} =0\}.  
      \]
      \item The equality $(0=)x_{1,1}=x_{1,2}$ gives the cells
      \[
      \qquad\qquad\quad 
      \gamma_3:=D^{+,+,+,\ldots,+}_{\ell+1,2,1,\ldots,1}(1,2,3,\ldots,k),\qquad 
      \gamma_4:=D^{+,-,+,\ldots,+}_{\ell+1,2,1,\ldots,1}(1,2,3,\ldots,k)
      \]
      that satisfy $\gamma_4=\varepsilon_2 \cdot\gamma_3$.
      Both cells belong to the linear subspace 
      \[
      \qquad\qquad V_2=\{(x_1,\ldots,x_k)\in \R^{(d+1)\times k} : x_{1,1}=\cdots=x_{\ell,1} =0,\,x_{1,1}=x_{1,2}\}.  
      \]
      \item The equality $x_{1,r-1}=x_{1,r}$ for $3\leq r\leq k$ gives cells:
      \[
         \gamma_{2r-1}:=D^{+,+,+,\ldots,+}_{\ell+1,\ldots,1,2,1,\ldots,1}(1,\ldots,r-1,r,r+1,\ldots,k),
      \]
      \[
         \gamma_{2r}:=D^{+,+,+,\ldots,+}_{\ell+1,\ldots,1,2,1,\ldots,1}(1,\ldots,r,r-1,r+1,\ldots,k)
      \]
      satisfying $\gamma_{2r}=\tau_{r-1,r}\cdot\gamma_{2r-1}$. 
      In these cells the index $2$ in the subscript $\ell+1,\ldots,1,2,1,\ldots,1$ appears at the position $r$.
	  These cells belong to the linear subspace 
      \[
      \qquad\qquad V_r=\{(x_1,\ldots,x_k)\in \R^{(d+1)\times k} : x_{1,1}=\cdots=x_{\ell,1} =0,\,x_{1,r-1}=x_{1,r}\}.  
      \]
      \end{compactenum}
      Again $e_{\theta}$ denotes a generator in $C_{N_2+1}(X_{d,k},X_{d,k}^{>1})$ that corresponds to the cell~$\theta$. 
      Furthermore $e_{\gamma_1},\ldots,e_{\gamma_{2k}}$ denote generators in $C_{N_2}(X_{d,k},X_{d,k}^{>1})$ related to the cells $\gamma_1,\ldots,\gamma_{2k}$.\newline
      As before, the boundary of the cell~$\theta$ is contained in the union of the linear subspaces $V_1,\ldots,V_k$.
      Thus we can orient cells $\gamma_{2i-1},\gamma_{2i}$ consistently with the orientation of $V_i$, $1\leq i\leq k$, that is given in such a way that 
      \[
      \partial e_{\theta}=(e_{\gamma_1}+e_{\gamma_2})+(e_{\gamma_3}+e_{\gamma_4})+\cdots + (e_{\gamma_{2k-1}}+e_{\gamma_{2k}}).
      \]
      Hence,
      \begin{equation}
      	\label{rel - 3}
      \qquad\partial e_{\theta}=(1+(-1)^{d-\ell}\varepsilon_1)\cdot e_{\gamma_1}+(1+(-1)^d\varepsilon_2)\cdot e_{\gamma_3} +\sum_{i=3}^{k}
      (1+(-1)^d\tau_{i-1,i})\cdot e_{\gamma_{2i-1}}.
      \end{equation}
\end{compactenum}
The relations \eqref{rel - 1}, \eqref{rel - 2} and \eqref{rel - 3} that we have now derived will be essential in the proofs of Theorems \ref{th-matrices_num_odd_equip_exist} and \ref{th-main:D(j,2)}.
	
\end{example}

\subsection{The arrangements parametrized by a cell}
\label{sec:parametrization}
In this section we describe all arrangements of $k$ hyperplanes parametrized by the cell 
\[
\theta:=D^{+,+,+,\ldots,+}_{\ell+1,1,1,\ldots,1}(1,2,3,\ldots,k),
\]
where $1\leq \ell\leq d-1$.
This description will be one of the key ingredients in Section~\ref{sec:proofs} when the obstruction cocycle is evaluated on the cell~$\theta$.

Recall that the cell~$\theta$ is defined as the intersection of the sphere $S(\R^{(d+1)\times k})$ and the cone given by the inequalities:
\[
0<_{\ell+1}x_1<_1 x_2<_1 \cdots <_1 x_k.
\] 
Consider the binomial coefficient moment curve $\hat{\gamma}\colon\R\longrightarrow\R^{d}$ defined by
\begin{equation}
\label{eq:moment-curve}
\hat{\gamma}(t)=(t,\tbinom{t}{2},\tbinom{t}{3},\ldots, \tbinom{t}{d} )^t.	
\end{equation}
After embedding $\R^d\longrightarrow\R^{d+1}, (\xi_1,\ldots,\xi_d)^t\longmapsto (1,\xi_1,\ldots,\xi_d)^t$ it corresponds to the curve $\gamma\colon\R\longrightarrow\R^{d+1}$ given
\[
\gamma(t)=(1,t,\tbinom{t}{2},\tbinom{t}{3},\ldots, \tbinom{t}{d} )^t.
\]
Consider the following points on the moment curve $\gamma$:
\begin{equation}
\label{eq:point_on_moment_curve}	 
q_1:=\gamma(0),\ldots,q_{\ell+1}:=\gamma(\ell).
\end{equation}
Next, recall that each oriented affine hyperplane $\hat{H}$ in $\R^d$ (embedded in $\R^{d+1}$) determines the unique linear hyperplane $H$ such that $\hat{H}=H\cap\R^d$, and almost vice versa.
Now, the family of arrangements parametrized by the (open) cell~$\theta$ is described as follows:

\begin{lemma}
\label{lem:parametrization-of-theta}
The cell~$\theta=D^{+,+,+,\ldots,+}_{\ell+1,1,1,\ldots,1}(1,2,3,\ldots,k)$ parametrizes all arrangements $\H=(H_1,\ldots,H_k)$ of $k$ linear hyperplanes in $\R^{d+1}$, where the order and orientation are fixed appropriately such that
\begin{compactitem}[ $\bullet$ ]
\item    $Q:=\{q_1,\ldots,q_{\ell}\}\subset H_1$,
\item    $q_{\ell+1}\notin H_1$, 
\item  $q_1\notin H_2,\ldots, q_1\notin H_k$, and
\item  $H_2,\dots,H_k$ have unit normal vectors with different (positive) first coordinates, that is,
       $|\{ \langle x_2,q_1\rangle, \langle x_3,q_1\rangle,\ldots,\langle x_k,q_1\rangle  \}|=k-1$.
\end{compactitem}
Here $x_i\in S(\R^{(d+1)\times k})$ is a unit normal vector of the hyperplane $H_i$, for $1\leq i\leq k$.
\end{lemma}

\begin{proof}
Observe that 
$\{q_1,\ldots,q_{\ell}\}\subset H_1$ holds if and only if $\langle x_1,q_1\rangle=\langle x_1,q_2 \rangle=\cdots=\langle x_1,q_{\ell}\rangle=0$ if and only if $x_{1,1}=x_{2,1}=\cdots=x_{\ell,1}=0$:
This is true since we have the binomial moment curve, so $q_i=\gamma(i-1)$ has only the first
$i$ coordinates non-zero.
\\
Furthermore, $q_{\ell+1}\notin H_1$ holds if and only if $x_{\ell+1,1}\neq 0$;
choosing an appropriate orientation for $H_1$ we can assume that $x_{\ell+1,1}>0$.
\\
The third condition is equivalent to 
$0\notin \{ \langle x_2,q_1\rangle, \langle x_3,q_1\rangle,\ldots,\langle x_k,q_1\rangle  \}$,
that is, $x_{1,2},x_{1,3},\dots,x_{1,k}\neq0$. Choosing orientations of $H_2,\dots,H_k$ suitably
this yields $x_{1,2},x_{1,3},\dots,x_{1,k}>0$.
\\
Since the values $x_{1,2}=\langle x_2,q_1\rangle$, $x_{1,3}= \langle x_3,q_1\rangle$, \ldots, 
$x_{1,k}=\langle x_k,q_1\rangle $ are positive and distinct, we get 
$0<x_{1,2}<x_{1,3}<\cdots<x_{1,k}$ by choosing the right order on $H_2,\ldots,H_k$. 
\end{proof}

\section{Proofs}
\label{sec:proofs}

\subsection{Proof of Theorem \ref{th-matrices_corr_equipartitions}}

Let $d\ge 1$, $j\geq 1$, $\ell\geq 0$ and $k\geq 2$ be integers with the property that $dk =j(2^k-1)+ \ell$ for $0\leq \ell\leq d-1$.

 Consider a collection of $j$ ordered disjoint intervals $\M=(I_1, \ldots, I_j)$ along the moment curve $\gamma$.
 Let $Q=\{q_1,\dots,q_\ell\} \subset \gamma$ be a set of $\ell$ predetermined points that lie to the left of the interval $I_1$.
 We prove Theorem~\ref{th-matrices_corr_equipartitions} in two steps.

 \begin{lemma} \label{lem:matrices_lead_to_equipartitions}
Let $A$ be an $\ell$-equiparting matrix, that is, a binary matrix of size $k \times j2^k$ with one row of transition count $d-\ell$ and all other rows of transition count $d$ such that $A = (A_1, \dots, A_j)$ for Gray codes $A_1,\ldots,A_j$ with the property that the last column of $A_i$ is equal to the first column of $A_{i+1}$ for $1 \le i < j$. 
 Then $A$ determines an arrangement $\H$ of $k$ affine hyperplanes that equipart $\M=(I_1, \ldots, I_j)$ and one of the  hyperplanes passes through each point in $Q$.	
 \end{lemma}
 \begin{proof}
 Without loss of generality we assume that the first row of the matrix $A$ has transition count $d-\ell$ while rows $2$ through $k$ have  transition count $d$.
For a row $a_s$ of the matrix $A$, denote by $t_s$ its transition count, $1\leq s\leq k$.
 
 Place $j(2^k +1)$ ordered points $q_{\ell+1},\dots, q_{\ell+j(2^k +1)}$ on $\gamma$, such that 
 \[
 I_i =[q_{\ell+(i-1)2^k + i},\, q_{\ell+i2^{k} +i}]
 \]
 and each sequence of $2^k+1$ points divides $I_i$ into $2^k$ subintervals of equal length. \emph{Ordered} refers to the property that $q_r=\gamma(t_r)$ if $t_1< t_2 < \dots <t_{j(2^k+1)}$.

We now define the hyperplanes in $\H$ by specifying which of the points they pass through and then choosing their orientations.
Force the affine hyperplane $H_1$ to pass through all of the points in~$Q$. 
For $s=1,\dots, i$, the affine hyperplane $H_s$ passes through $x_{\ell+ r+i}$ if there is a bit change in row $a_s$ from entry $r$ to entry $r+1$ for $(i-1)2^k < r \le i2^k$.  
Orient $H_s$ such that the subinterval $[q_r,q_{r+1}]$ is on the positive side of $H_s$ if it corresponds to a $0$-entry in~$a_s$. Since each $A_1, \dots, A_j$ is a Gray code, the arrangement $\H$ is indeed an equipartition.
\vspace{-3pt}
\end{proof}

\begin{lemma} \label{lem:equip_lead_to_matrices}
Every arrangement of $k$ affine hyperplanes $\H$ that equiparts $\M=(I_1, \ldots, I_j)$ and where one of the  hyperplanes passes through each point of $Q$ induces a unique binary matrix $A$ as in Lemma~\ref{lem:matrices_lead_to_equipartitions}.	
\end{lemma}

\begin{proof}
Since $dk =j(2^k-1)+ \ell$ and $0\leq \ell\leq d-1$, the hyperplanes in $\H$ must pass through the points  $q_{\ell+(i-1)2^k + i +1}, \dots, q_{\ell+i2^{k} +i-1}$ of the intervals $I_i$ for $i\in\{1, \dots, j\}$. Recording the position of the subintervals $[q_{\ell+r},q_{\ell+r+1}]$, for $r\neq i2^k +i$, with respect to each hyperplane leads to a matrix as in described in Lemma~\ref{lem:matrices_lead_to_equipartitions}.	
\end{proof}
 
\begin{figure}[ht!]
\centering
\includegraphics[scale=1.2]{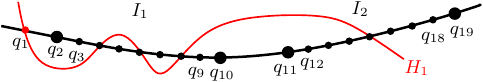}
\label{fig_mom_curv_2}
\caption{\small Illustration of one step in the proof of Lemma~\ref{lem:matrices_lead_to_equipartitions}. 
Here~$H_1$ is an affine hyperplane bisecting two intervals $I_1$ and $I_2$ on the $5$-dimensional moment curve.}
\end{figure}

Thus the number of non-equivalent $\ell$-equiparting matrices is the same as the number of arrangements of $k$ affine hyperplanes $\H$ that equipart the collection of $j$ disjoint intervals on the moment curve in $\R^d$, up to renumbering and orientation change of hyperplanes in $\H$, when one of the hyperplanes is forced to pass through $\ell$ prescribed points on the moment curve lying to the left of the intervals. This concludes the proof of Theorem~\ref{th-matrices_corr_equipartitions}. \qed

\subsection{Proof of Theorem \ref{th-matrices_num_odd_equip_exist}}

Let $j\geq 1$ and $k\geq 3$ be integers with $d=\lceil \tfrac{2^k-1}{k}j\rceil$ and $\ell = dk - (2^k-1)j$.
In addition, assume that the number of non-equivalent $\ell$-equiparting matrices of size $k \times j2^k$ is odd.

In order to prove that $\Delta(j,k) \le d$ it suffices by Theorem~\ref{prop:CS/TM-1} to prove that there is no $\Wk_k$-equivariant map
\[
X_{d,k} \longrightarrow S(W_k\oplus U_{k}^{\oplus j}),
\]
whose restriction to $X_{d,k}^{>1}$ is $\Wk_k$-homotopic to $\nu\circ\psi_{\mathcal{M}}|_{X_{d,k}^{>1}}$ for $\mathcal{M}=(I_1, \ldots, I_j)$.
Following Section~\ref{subsec:obstruction theory} we verify that the cohomology class
\[
[\oo(g)]\in \mathcal{H}_{\Wk_k}^{N_2+1}\big(X_{d,k},X_{d,k}^{>1} \,; \,\pi_{N_2}(S(W_k\oplus U_{k}^{\oplus j}))\big),
\]
does not vanish, where $g=\nu\circ\psi_{\M}|_{X^{(N_2)}}$.

Consider the cell $\theta:=D^{+,+,+,\ldots,+}_{\ell+1,1,1,\ldots,1}(1,2,3,\ldots,k)$ of dimension $(d+1)k-1-\ell=N_2+1$ in $X_{d,k}$, as in Example \ref{example:boundary-1}.
Let $e_{\theta}$ denote the corresponding basis element of the cell~$\theta$ in the cellular chain group $C_{N_2+1}(X_{d,k},X_{d,k}^{>1})$, and let $h_{\theta}$ be the attaching map of $\theta$.
This cell is cut out from the unit sphere $S(\R^{(d+1)\times k})$ by the following inequalities:
\[
0<_{\ell+1}x_1<_1 x_2<_1 \cdots <_1 x_k.
\] 
In particular, this means that the first $\ell$ coordinates of $x_1$ are zero, i.e., $x_{1,1}=x_{2,1}=x_{3,1}=\cdots =x_{\ell,1}=0$, and $x_{\ell+1,1}>0$.

Let us fix $\ell$ points $Q=\{q_1,\dots,q_\ell\}$ on the moment curve \eqref{eq:moment-curve} in $\R^{d+1}$ as it was done in \eqref{eq:point_on_moment_curve}: $q_1:=\gamma(0),\ldots,q_{\ell}:=\gamma(\ell-1)$.
Then, by Lemma~\ref{lem:parametrization-of-theta}, the relative interior of $D^{+,+,+,\ldots,+}_{\ell+1,1,1,\ldots,1}(1,2,3,\ldots,k)$ parametrizes the arrangements $\H=(H_1,\ldots,H_k)$ for which orientations and order of the hyperplanes are fixed with $H_1$ containing all the points from $Q$.
According to the formula \eqref{eq:obstruction_cocycle} we have that
\[
\oo(g)(e_{\theta}) = [\nu\circ\psi_{\M}\circ h_{\theta}|_{\partial\theta}]
 = \sum\deg(\nu\circ\psi_{\M}|_{X^{(N_2+1)}}\circ h_{\theta}|_{S_i})\cdot\zeta,
\]
where as before $\zeta\in \pi_{N_2}(S(W_k\oplus U_{k}^{\oplus j}))\cong\Z$ is a generator, and the sum ranges over all arrangements of $k$ hyperplanes in $\relint\theta$ that equipart~$\M$.
Here, as before, $S_i$ denotes a small $N_2$-sphere around a root of the function $\psi_{\M}|_{X^{(N_2+1)}}\circ h_{\theta}$, i.e., the point that parametrizes an arrangements of $k$ hyperplanes in $\relint\theta$ that equipart~$\M$.\newline
Now, the local degrees of the function $\nu\circ\psi_{\M}|_{X^{(N_2+1)}}\circ h_{\theta}$ are $\pm 1$.
Indeed, in a small neighborhood $U\subseteq \relint\theta$ around any root the test map $\psi_{\M}$ is a continuous bijection.
Thus $\psi_{\M} |_{\partial U}$ is a continuous bijection into some $N_2$-sphere around the origin in $W_k\oplus U_{k}^{\oplus j}$ and by compactness of $\partial U$ is a homeomorphism.
Consequently, 
\begin{equation}
	\label{eq:eval_obstruction_cocycle_1}
	\oo(g)(e_{\theta}) = \sum\deg(\nu\circ\psi_{\M}|_{X^{(N_2+1)}}\circ h_{\theta}|_{S_i})\cdot\zeta
	=\big(\sum\pm 1\big)\cdot\zeta = a\cdot\zeta,
\end{equation}
where the sum ranges over all arrangements of $k$ hyperplanes in $\relint\theta$ that equipart $\M$.
According to Theorem~\ref{th-matrices_corr_equipartitions} the number of $(\pm 1)$'s in the sum \eqref{eq:eval_obstruction_cocycle_1} is equal to the number of non-equivalent $\ell$-equiparting matrices of size $k \times j2^k$.
By our assumption this number is odd and consequently $a\in\Z$ is an odd integer.
We obtained that
\begin{equation}
	\label{eq:eval_obstruction_cocycle_2}
	\oo(g)(e_{\theta}) = a\cdot\zeta,
\end{equation}
where $a\in\Z$ is an odd integer.

\begin{remark}\label{rem:k=2}
It is important to point out that the calculations and formulas up to this point also hold for $k=2$. 
The assumption $k\geq 3$ affects the $\Wk_k=(\Z/2)^k\rtimes\Sym_k$ module structure on $\pi_{N_2}(S(W_k\oplus U_{k}^{\oplus j}))\cong\Z$.    
For $k\geq2$ every generator $\varepsilon_i$ of the subgroup $(\Z/2)^k$ acts trivially, while each transposition $\tau_{i,t}$, a generator of the subgroup $\Sym_k$,  acts as multiplication by $-1$ in the case $k\geq 3$,  and as multiplication by $(-1)^{j+1}$ in the case $k=2$. 
\end{remark}

Finally, we prove that $[\oo(g)]$ does not vanish and conclude the proof.
This will be achieved by proving that  the cocycle $\oo(g)$ is not a coboundary. \newline
Let us assume to the contrary that $\oo(g)$ is a coboundary. 
Thus there exists a cochain 
\[
\bb\in \mathcal{C}_{\Wk_k}^{N_2}\big(X_{d,k},X_{d,k}^{>1} \,; \,\pi_{N_2}(S(W_k\oplus U_{k}^{\oplus j}))\big)
\]
such that $\oo(g)=\delta\bb$, where $\delta$ denotes the coboundary operator.
In the case when 
\begin{compactenum}[\rm (1)]
\item $\ell=0$ the relation \eqref{rel - 1} implies that \allowdisplaybreaks 
\begin{eqnarray*}
a\cdot \zeta &=& \oo(g)(e_{\theta})=\delta\bb(e_{\theta}) =\bb(\partial e_{\theta})\\
&=&(1+(-1)^d\varepsilon_1)\cdot \bb(e_{\gamma_1})+\sum_{i=2}^{k}
      (1+(-1)^d\tau_{i-1,i})\cdot \bb(e_{\gamma_{2i-1}}) \\
&=&  (1+(-1)^d)\cdot \bb(e_{\gamma_1})+\sum_{i=2}^{k}
      (1+(-1)^{d+1})\cdot \bb(e_{\gamma_{2i-1}})\\
&=& 2b\cdot\zeta,          
\end{eqnarray*}	
for some integer $b$.
Since $a$ is an odd integer this is not possible, and therefore $\oo(g)$ is not a coboundary.
\item $\ell=1$ the relation \eqref{rel - 2} implies that \allowdisplaybreaks 
\begin{eqnarray*}
a\cdot \zeta &=&\oo(g)(e_{\theta})=\delta\bb(e_{\theta}) =\bb(\partial e_{\theta})\\
&=&	(1+(-1)^{d-1}\varepsilon_1)\cdot \bb(e_{\gamma_1})+\\
& & (1+(-1)^d\varepsilon_2+(-1)^d\tau_{1,2}+(-1)^{d+d}\varepsilon_1\tau_{1,2})\cdot \bb(e_{\gamma_3})+ \\
& & \sum_{i=3}^{k} (1+(-1)^d\tau_{i-1,i})\cdot \bb(e_{\gamma_{2i-1}})\\
&=& (1+(-1)^{d-1})\cdot \bb(e_{\gamma_1})+(1+(-1)^d+(-1)^{d+1}-1)\cdot \bb(e_{\gamma_3})+\\
& & \sum_{i=3}^{k} (1+(-1)^{d+1})\cdot \bb(e_{\gamma_{2i-1}})\\
&=& (1+(-1)^{d-1})\cdot \bb(e_{\gamma_1})+\sum_{i=3}^{k} (1+(-1)^{d+1})\cdot \bb(e_{\gamma_{2i-1}})\\
&=& 2b\cdot\zeta,
\end{eqnarray*}	
for $b\in\Z$.
Again we reached a contradiction, so $\oo(g)$ is not a coboundary.
\item $2\leq\ell\leq d-1$ the relation \eqref{rel - 3} implies that \allowdisplaybreaks 
\begin{eqnarray*}
a\cdot \zeta &=& \oo(g)(e_{\theta})=\delta\bb(e_{\theta}) =\bb(\partial e_{\theta})\\
&=&(1+(-1)^{d-\ell}\varepsilon_1)\cdot \bb(e_{\gamma_1})+(1+(-1)^d\varepsilon_2)\cdot \bb(e_{\gamma_3})+\\
& &\sum_{i=3}^{k}(1+(-1)^d\tau_{i-1,i})\cdot \bb(e_{\gamma_{2i-1}}) \\
&=&(1+(-1)^{d-\ell})\cdot \bb(e_{\gamma_1})+(1+(-1)^d)\cdot \bb(e_{\gamma_3})+\\
& &\sum_{i=3}^{k}(1+(-1)^{d+1})\cdot \bb(e_{\gamma_{2i-1}}) \\
&=& 2b\cdot\zeta,          
\end{eqnarray*}		
for an integer $b$.
Since $a$ is an odd integer this is not possible.
Again, $\oo(g)$ is not a coboundary.  \qed
\end{compactenum}

\subsection{Proof of Theorem~\ref{th-main:D(j,2)}} 

Let $j\geq 1$ be an integer with $d=\lceil \tfrac{3}{2}j\rceil$ and $\ell = 2d - 3j\leq 1$.

The proof of this theorem is done in the footsteps of the proof of Theorem~\ref{th-matrices_num_odd_equip_exist}.
In all three cases we rely on Theorem~\ref{prop:CS/TM-1} and prove
\begin{compactitem}[ $\bullet$ ]
\item the non-existence of $\Wk_2$-equivariant map 	$X_{d,2} \longrightarrow S(W_2\oplus U_{2}^{\oplus j})$ whose restriction to $X_{d,2}^{>1}$ is $\Wk_2$-homotopic to $\nu\circ\psi_{\mathcal{M}}|_{X_{d,2}^{>1}}$ for $\mathcal{M}=(I_1, \ldots, I_j)$; by
\item evaluating the obstruction cocycle $\oo(g)$ for $g=\nu\circ\psi_{\M}|_{X^{(N_2)}}$ on cells $D^{+,+}_{1,1}(1,2)$ or $D^{+,+}_{2,1}(1,2)$, depending on $\ell$ being $0$ or $1$, using Theorem~\ref{th-matrices_corr_equipartitions}; and then
\item prove that the cocycle $\oo(g)$ cannot be a coboundary, utilizing boundary formulas from Example~\ref{example:boundary-1}. 
\end{compactitem}

\subsubsection{$2$-bit Gray codes}
In order to evaluate the obstruction cocycle $\oo(g)$ on the relevant cells in the case $k=2$ we need to understand $(2 \times 4)$-Gray codes. 
These correspond to equipartitions of an interval $I$ on the moment curve into four equal orthants by intersecting with two hyperplanes $H_1$ and $H_2$ in altogether three points of the interval. 
There are two such configurations: either $H_1$ cuts through the midpoint of $I$ and $H_2$ separates
both halves of $I$ into equal pieces by two additional intersections, or the roles of $H_1$ and $H_2$ are reversed.
In terms of Gray codes we can express this as follows.

\begin{lemma}
	There are two different $2$-bit Gray codes that start with the zero column (or any other fixed binary vector of length~$2$):
\[
\begin{pmatrix}
0 & 1 & 1 & 0  \\
0 & 0 & 1 & 1  
\end{pmatrix}
\qquad
\text{and}
\qquad
\begin{pmatrix}
0 & 0 & 1 & 1  \\
0 & 1 & 1 & 0  
\end{pmatrix}
\]
\end{lemma}

\begin{proof}
The second column of the Gray code determines the rest of the code, and there are only two choices for a bit flip.	
\end{proof}

This means that in the case $k=2$ an $\ell$-equiparting matrix~$A$ has a more compact
representation: it is determined by the first column -- a binary vector of length~$2$ -- and $j$ additional bits, one for 
each~$A_i$, encoding whether the first bit flip in $A_i$ is in the first or second row. 
These $j$ bits {\em cannot be chosen independently} since there are restrictions imposed by the transition count.

\begin{lemma}
\label{lem:equipartin matrices - 2}
Let $j\geq 1$ be an integer with $d=\lceil \tfrac{3}{2}j\rceil$ and $\ell = 2d - 3j\leq 1$.	
\begin{compactenum}[\rm (1)]
\item If $\ell=0$, then the number of non-equivalent $0$-equiparting matrices is equal to
\[ \tfrac12\binom{j}{\tfrac{j}2}.\]
\item If $\ell=1$, then the number of non-equivalent $1$-equiparting matrices is equal to
\[ \binom{j}{\tfrac{j+1}2}.\]
\end{compactenum}
\end{lemma}
\begin{proof}
We count the number of non-equivalent $\ell$-equiparting matrices of the form $A=(A_1,\ldots,A_j)$ where $A_i$ is a $2$-bit Gray code.
A $(2\times 4)$-Gray code with the first bit flip in the first row has in total two bit flips in the first row and one bit flip in the second row. 

\noindent
{\bf (1)} Let $\ell=0$. 
Then $2d=3j$ and consequently $j$ has to be even.
The matrix $A$ must have transition count $d$  in each row. 
Thus, half of the $A_i$'s have the first bit flip in the first row. 
Consequently, $0$-equiparting matrices~$A$ with a fixed first column are in bijection with $\tfrac{j}2$-element subsets of a set with $j$ elements. 
By inverting the bits in each row we can fix the first column of $A$ to be the zero vector. 
Additionally, we are allowed to interchange the rows. 
Up to this equivalence there are $\frac{1}{2}\binom{j}{j/2}$
such matrices. 

\noindent
{\bf (2)} Let $\ell=1$. 
Then $2d=3j+1$ and so $j$ is odd.
The matrix $A$ must have transition count $d$ in one row while transition count $d-1$ in the remaining row.
Without loss of generality we can assume that $A$ have transition count $d$ in the first row.
Assume that $r$ of the $A_i$'s have the first bit flip in the first row.
Consequently, $j-r$ of the $A_i$'s have the first bit flip in the second row.
Now the transition count of the first row is $2r+j-r$ while the transition count of the second row is $r+2(j-r)$.
The system of equations $2r+j-r=d,\, r+2(j-r)=d-1$ yields that $r=\tfrac{j+1}{2}$.
Therefore, up to equivalence, there are $\binom{j}{r}$ such matrices.
\end{proof}

\subsubsection{The case $\ell=0\Leftrightarrow 2d=3j$}
Let $\theta:=D^{+,+}_{1,1}(1,2)$, and let  $e_{\theta}$ denote the related basis element of the cell~$\theta$ in the top cellular chain group $C_{2d+1}(X_{d,2},X_{d,2}^{>1})$ which, in this case, is equivariantly generated by $\theta$.
According to equation \eqref{eq:eval_obstruction_cocycle_1}, which also holds for $k=2$ as explained in Remark~\ref{rem:k=2}, 
\begin{equation}
	\label{eq:eval_obstruction_cocycle_3}
	\oo(g)(e_{\theta}) =\big(\sum\pm 1\big)\cdot\zeta = a\cdot\zeta,
\end{equation}
where $\zeta\in \pi_{2d+1}(S(W_2\oplus U_{2}^{\oplus j}))\cong\Z$ is a generator, and the sum ranges over all arrangements of two hyperplanes in $\relint\theta$ that equipart~$\M$.
Since $\theta$ parametrizes all arrangements $\H=(H_1,H_2)$ where orientations and order of hyperplanes are fixed, the sum in \eqref{eq:eval_obstruction_cocycle_3} ranges over all arrangements of two hyperplanes that equipart $\M$ where orientation and order of hyperplanes are fixed.
Therefore, by Theorem~\ref{th-matrices_corr_equipartitions}, the number of $(\pm 1)$'s in the sum of \eqref{eq:eval_obstruction_cocycle_3} is equal to the number of non-equivalent $0$-equiparting matrices of size $2\times 4j$.
Now, Lemma~\ref{lem:equipartin matrices - 2} implies that the number of $(\pm 1)$'s in the sum of \eqref{eq:eval_obstruction_cocycle_3} is $\frac{1}{2}\binom{j}{j/2}$.
Consequently, integer $a$ is odd if and only if $\frac{1}{2}\binom{j}{j/2}$ is odd.

Assume that the cocycle $\oo(g)$ is a coboundary.
Hence, there exists a cochain 
\[
\bb\in \mathcal{C}_{\Wk_2}^{2d}\big(X_{d,2},X_{d,2}^{>1} \,; \,\pi_{2d}(S(W_2\oplus U_{2}^{\oplus j}))\big)
\]
with the property that $\oo(g)=\delta\bb$.
The relation \eqref{rel - 1} for $k=2$ transforms into 
\[
      \partial e_{\theta}=(1+(-1)^d\varepsilon_1)\cdot e_{\gamma_1}+(1+(-1)^d\tau_{1,2})\cdot e_{\gamma_{3}}.
\]
Thus we have that\allowdisplaybreaks 
\begin{eqnarray*}
a\cdot \zeta &=& \oo(g)(e_{\theta})=\delta\bb(e_{\theta}) =\bb(\partial e_{\theta})\\
&=& (1+(-1)^d\varepsilon_1)\cdot \bb(e_{\gamma_1})+(1+(-1)^d\tau_{1,2})\cdot \bb(e_{\gamma_{3}}) \\
&=& (1+(-1)^d)\cdot \bb(e_{\gamma_1})+(1+(-1)^{d+j+1})\cdot \bb(e_{\gamma_{3}}) \\
&=& 2b\cdot\zeta.          
\end{eqnarray*}	
Consequently, $\oo(g)$ is not a coboundary if and only if $a$ is odd if and only if $\frac{1}{2}\binom{j}{j/2}$ is odd.
Having in mind the Kummer criterion stated below we conclude that: A $\Wk_2$-equivariant map
$X_{d,2} \longrightarrow S(W_2\oplus U_{2}^{\oplus j})$ whose restriction to $X_{d,2}^{>1}$ is $\Wk_2$-homotopic to $\nu\circ\psi_{\mathcal{M}}|_{X_{d,2}^{>1}}$ does not exists {\em if and only is} $\oo(g)$ is not a coboundary {\em if and only if} $a$ is an odd integer {\em if and only if} $\frac{1}{2}\binom{j}{j/2}$ is odd {\em if and only if} $j=2^t$ for $t\geq1$. 

\begin{lemma}[Kummer \cite{kummer1852}]
\label{lemma:kummer}
Let $n \ge m \ge 0$ be integers and let $p$ be a prime. The maximal integer $k$ such that $p^k$ divides $\binom{n}{m}$ is the number of carries when $m$ and $n-m$ are added in base $p$.
\end{lemma}

Thus we have proved the case (ii) of Theorem~\ref{th-main:D(j,2)}.
Moreover, since the primary obstruction $\oo(g)$ is the only obstruction, we have proved that a $\Wk_2$-equivariant map $X_{d,2} \longrightarrow S(W_2\oplus U_{2}^{\oplus j})$ whose restriction to $X_{d,2}^{>1}$ is $\Wk_2$-homotopic to $\nu\circ\psi_{\mathcal{M}}|_{X_{d,2}^{>1}}$ exists if and only if $j$, an even integer, is not a power of $2$.

\subsubsection{The case $\ell=1\Leftrightarrow 2d=3j+1$}
Let $\theta:=D^{+,+}_{2,1}(1,2)$, and again let $e_{\theta}$ denote the related basis element of the cell~$\theta$ in the cellular chain group $C_{2d}(X_{d,2},X_{d,2}^{>1})$ which, in this case, is equivariantly generated by two cells $D^{+,+}_{2,1}(1,2)$ and $D^{+,+}_{1,2}(1,2)$.
Again, the equation \eqref{eq:eval_obstruction_cocycle_1} implies that
\begin{equation}
	\label{eq:eval_obstruction_cocycle_4}
	\oo(g)(e_{\theta}) =\big(\sum\pm 1\big)\cdot\zeta = a\cdot\zeta,
\end{equation}
where $\zeta\in \pi_{2d+1}(S(W_2\oplus U_{2}^{\oplus j}))\cong\Z$ is a generator, and the sum ranges over all arrangements of $k$ hyperplanes in $\relint\theta$ that equipart~$\M$.
The cell~$\theta$ parametrizes all arrangements $\H=(H_1,H_2)$ where $H_1$ passes through the given point on the moment curve and orientations and order of hyperplanes are fixed.
Thus, the sum in \eqref{eq:eval_obstruction_cocycle_4} ranges over all arrangements of two hyperplanes that equipart $\M$ where $H_1$ passes through the given point on the moment curve with order and orientation of hyperplanes being fixed.
Therefore, by Theorem~\ref{th-matrices_corr_equipartitions}, the number of $(\pm 1)$'s in the sum of \eqref{eq:eval_obstruction_cocycle_4} is the same as the number of non-equivalent $1$-equiparting matrices of size $2\times 4j$.
Again, Lemma~\ref{lem:equipartin matrices - 2} implies that the number of $(\pm 1)$'s in the sum of \eqref{eq:eval_obstruction_cocycle_4} is $\binom{j}{(j+1)/2}$.
The integer $a$ is odd if and only if $\binom{j}{(j+1)/2}$ is odd if and only if $j=2^t-1$ for $t\geq 1$.

Assume that the cocycle $\oo(g)$ is a coboundary.
Then there exists a cochain 
\[
\bb\in \mathcal{C}_{\Wk_2}^{2d-1}\big(X_{d,2},X_{d,2}^{>1} \,; \,\pi_{2d-1}(S(W_2\oplus U_{2}^{\oplus j}))\big)
\]
with the property that $\oo(g)=\delta\bb$.
Now, the relation \eqref{rel - 2} for $k=2$ transforms into 
\[
\partial e_{\theta} = (1+(-1)^{d-1}\varepsilon_1)\cdot e_{\gamma_1}+ (1+(-1)^d\varepsilon_2+(-1)^d\tau_{1,2}+(-1)^{d+d}\varepsilon_1\tau_{1,2})\cdot e_{\gamma_3}.
\]
Thus, having in mind that $j$ has to be odd, we have \allowdisplaybreaks 
\begin{eqnarray}
\label{rel - 6}
a\cdot \zeta &=& \oo(g)(e_{\theta})=\delta\bb(e_{\theta}) =\bb(\partial e_{\theta})\nonumber \\
&=& (1+(-1)^{d-1}\varepsilon_1)\cdot \bb(e_{\gamma_1})+ \nonumber \\
& & (1+(-1)^d\varepsilon_2+(-1)^d\tau_{1,2}+(-1)^{d+d}\varepsilon_1\tau_{1,2})\cdot \bb(e_{\gamma_3})\nonumber \\
&=& (1+(-1)^{d-1})\cdot \bb(e_{\gamma_1})+ (1+(-1)^d+(-1)^{d+j+1}+(-1)^{j+1})\cdot \bb(e_{\gamma_3})\nonumber \\
&=& (1+(-1)^{d-1})\cdot \bb(e_{\gamma_1})+ (1+(-1)^d+(-1)^{d}+1)\cdot \bb(e_{\gamma_3})\nonumber \\
&=& 
\begin{cases}
2\bb(e_{\gamma_1}),  & d\text{ odd}\\
4\bb(e_{\gamma_3}),  & d\text{ even}.	
\end{cases} 
\end{eqnarray}
Now, we separately consider cases depending on parity of $d$ and value of $j$. 

\smallskip
{\bf (1)} Let $d$ be odd.
Recall that $a$ is odd if and only if $j=2^t-1$ for $t\geq 1$. 
Since $d=\tfrac12(3j+1)=3\cdot 2^{t-1}-1$ and $d$ is odd we have that for $j=2^t-1$, with $t\geq 2$, the integer $a$ is odd and consequently $\oo(g)$ is not a coboundary.
Thus a $\Wk_2$-equivariant map $X_{d,2} \longrightarrow S(W_2\oplus U_{2}^{\oplus j})$ whose restriction to $X_{d,2}^{>1}$ is $\Wk_2$-homotopic to $\nu\circ\psi_{\mathcal{M}}|_{X_{d,2}^{>1}}$ does not exists.
We have proved the case (ii) of Theorem~\ref{th-main:D(j,2)} for $t\geq 2$.

\smallskip
{\bf (2)} Let $d=2$ and $j=1=2^1-1$.
Then the integer $a$ is again odd and consequently cannot be divisible by $4$ implying again that $\oo(g)$ is not a coboundary.
Therefore a $\Wk_2$-equivariant map $X_{2,2} \longrightarrow S(W_2\oplus U_{2})$ whose restriction to $X_{2,2}^{>1}$ is $\Wk_2$-homotopic to $\nu\circ\psi_{\mathcal{M}}|_{X_{2,2}^{>1}}$ does not exists.
This concludes the proof of the case (ii) of Theorem~\ref{th-main:D(j,2)}.

\smallskip
{\bf (3)} Let $d\geq 4$ be even.
Now we determine the integer $a$ by computing local degrees $\deg(\nu\circ\psi_{\M}|_{X^{(N_2+1)}}\circ h_{\theta}|_{S_i})$; see  
\eqref{eq:eval_obstruction_cocycle_1} and \eqref{eq:eval_obstruction_cocycle_4}.
We prove, almost identically as in \cite[Proof of Lem.\,5.6]{bfhz_2015}, that all local degrees equal, either $1$ or $-1$. 

That local degrees of $\nu\circ\psi_{\M}|_{\theta}$ are $\pm 1$ is simple to see since in a small neighborhood $U$ in $\relint\theta$ around any root $\lambda u+ (1-\lambda)v$ the test map $\psi_{\M}|_{\theta}$ is a continuous bijection.
Indeed, for any vector $w \in W_2\oplus U_{2}^{\oplus j}$, with sufficiently small norm, there is exactly one $\lambda u'+ (1-\lambda)v' \in U$ with $\psi_{\M}(\lambda u'+ (1-\lambda)v')= w$. 
Thus $\psi_{\M}|_{\partial U}$ is a continuous bijection into some $3j$-sphere around the origin of $W_2\oplus U_{2}^{\oplus j}$ and by compactness of $\partial U$ is a homeomorphism.

Next we compute the signs of the local degrees.
First we describe a neighborhood of every root of the test map $\psi_{\M}$ in $\relint\theta$.
Let $\lambda u+ (1-\lambda)v \in \relint\theta$ with $\psi_{\M}(\lambda u+ (1-\lambda)v) = 0$. 
Consequently $\lambda=\tfrac12$.
Denote the intersections of the hyperplane $H_u$ with the moment curve by $x_1, \ldots, x_d$ in the correct order along the moment curve. 
Similarly, let $y_1, \ldots, y_d$ be the intersections of $H_v$ with the moment curve. 
In particular, $x_1$ is the point $q_1$ that determines the cell~$\theta$, see Lemma~\ref{lem:parametrization-of-theta}. 
Choose an $\epsilon > 0$ such that $\epsilon$-balls around $x_2, \ldots, x_d$ and around $y_1, \ldots, y_d$ are pairwise disjoint with the property that these balls intersect the moment curve only in precisely one of the intervals $I_1,\ldots,I_j$.
Pairs of hyperplanes $(H_{u'}, H_{v'})$ with $\lambda u'+ (1-\lambda)v' \in \relint\theta$ that still intersect the moment curve in the corresponding $\epsilon$-balls parametrize a neighborhood of $\tfrac12 u+ \tfrac12 v$. 
The local neighborhood consisting of pairs of hyperplanes with the same orientation still intersecting the moment curve in the corresponding $\epsilon$-balls where the parameter $\lambda$ is in some neighborhood of $\tfrac12$.
For sufficiently small $\epsilon>0$ the neighborhood can be naturally parametrized by the product
\[
(\tfrac12-\epsilon,\tfrac12+\epsilon )\times \prod_{i=2}^{2d} (-\epsilon, \epsilon),
\]
where the first factor relates to $\lambda$, the next $d-1$ factors correspond to neighborhoods of the $x_2, \ldots, x_d$ and the last $d$ factors to $\epsilon$-balls around $y_1, \ldots, y_d$. 
A natural basis of the tangent space at $\tfrac12 u+ \tfrac12 v$ is obtained via the push-forward of the canonical basis of $\R^{2d}$ as tangent space at $(\frac12,0,\dots,0)^t$. 

Consider the subspace $Z \subseteq \relint\theta$ that consists all points $\lambda u+(1-\lambda)v$ associated to the pairs of hyperplanes $(H_u, H_v)$ such that both hyperplanes intersect the moment curve in $d$ points. 
In the space $Z$ the local degrees only depend on the orientations of the hyperplanes $H_u$ and~$H_v$, but these are fixed since $Z\subseteq \relint\theta$. 
Indeed, 
any two neighborhoods of distinct roots of the test map $\psi_{\M}$ can be mapped onto each other by a composition of coordinate charts since their domains coincide. 
This is a smooth map of degree $1$: the Jacobian at the root is the identity map. 
Let $\tfrac12 u+ \tfrac12 v$ and $\tfrac12  u'+\tfrac12 v'$ be roots in $Z$ of the test map $\psi_{\M}$ and let $\Psi$ be the change of coordinate chart described above. 
Then $\psi_{\M}$ and $\psi_{\M}\circ\Psi$ differ in a neighborhood of $\tfrac12  u+\tfrac12 v$ just by a permutation of coordinates. 
This permutation is always even by the following: 
\begin{claim*}
Let $A$ and $B$ be finite sets of the same cardinality. 
Then the cardinality of the symmetric sum $A \vartriangle B$ is even.
\end{claim*}

\noindent
The orientations of the hyperplanes $H_u$ and $H_v$ are fixed by the condition that $\tfrac12 u+ \tfrac12 v \in\relint\theta$.
Thus, $H_u$ and $H_v$ are completely determined by the set of intervals that $H_u$ cuts once. 
Let $A \subseteq \{1, \dots, j\}$ be the set of indices of intervals $I_1,\ldots,I_h$ that $H_u$ intersects once, and let $B \subseteq \{1, \dots, j\}$ be the same set for $H_v$. 
Then $\Psi$ is a composition of a multiple of $A \vartriangle B$ transpositions and, hence, an even permutation.
This means that all the local degrees ($\pm 1$'s) in the sum \eqref{eq:eval_obstruction_cocycle_4} are of the same sign, and consequently $a=\pm \binom{j}{(j+1)/2}$.

Now, since $d$ is even the equality \eqref{rel - 6} implies that 
\[
a\cdot\zeta = 4b\cdot\zeta.
\]
Thus, if $\oo(g)$ is a coboundary $a$ is divisible by $4$.
In the case $j=2^t+1$ where $t\geq2$, and $d=3\cdot 2^{t-1}+2$ the Kummer criterion implies that the binomial coefficient $\binom{j}{(j+1)/2}$ is divisible by $2$ but {\em not} by $4$.
Hence, $\oo(g)$ is not a coboundary and a $\Wk_2$-equivariant map $X_{d,2} \longrightarrow S(W_2\oplus U_{2}^{\oplus j})$ whose restriction to $X_{d,2}^{>1}$ is $\Wk_2$-homotopic to $\nu\circ\psi_{\mathcal{M}}|_{X_{d,2}^{>1}}$ does not exist.

This concludes the final instance (iii) of  Theorem~\ref{th-main:D(j,2)}.\qed

\subsection{Proof of Theorem~\ref{th-main:D(j,3)}}

We prove both instances of the Ramos conjecture  $\Delta(2,3)=5$ and $\Delta(4,3)=10$ using Theorem \ref{th-matrices_num_odd_equip_exist}.
Thus in order to prove that
\begin{compactitem}[ $\bullet$ ]
\item 	$\Delta(2,3)=5$ it suffices to show that the number of non-equivalent $1$-equiparting matrices of size $3\times 2\cdot 2^3$ is odd, Proposition~\ref{prop:2_3_enumeration};
\item $\Delta(4,3)=10$ it suffices to show that the number of non-equivalent $2$-equiparting matrices of size $3\times 4\cdot 2^3$ is also odd, Enumeration~\ref{enum:4_3_enumeration}.
\end{compactitem}
Consequently we turn our attention to $3$-bit Gray codes.
It is not hard to see that the following lemma holds.

\begin{lemma}\label{lem:3BitGrayCodes}
Let  $c_1 \in \{0,1\}^3$ be a choice of first column. 
\begin{compactenum}[\normalfont (i)]
\item \label{lem:3BitGrayCodes_i} There are $18$ different $3$-bit Gray codes $A=(c_1,c_2, \dots,c_8) \in \{0,1\}^{3 \times 8}$ that start with $c_1$. They have transition counts $(3,2,2), \, (3,3,1)$, or $(4,2,1)$. 
\item \label{lem:3BitGrayCodes_ii} There are $3$ equivalence classes of Gray codes that start with with~$c_1$. 
    The three classes can be distinguished by their transition counts.
\end{compactenum}
\end{lemma}

\begin{proof}
(i): Starting at a given vertex of the $3$-cube, there are precisely $18$ Hamiltonian paths. 
This can be seen directly or by computer enumeration. \newline
(ii): Follows directly from (i), as all equivalence classes have size $6$:
      If $c_1=(0,0,0)^t$ then all elements in a class are obtained by permutation of rows. 
      For other choices of $c_1$, they are obtained by arbitrary permutations of rows 
      followed by the “correct” row bit-inversions to obtain $c_1$ in the first column.
\end{proof}

\begin{proposition}\label{prop:2_3_enumeration}
There are $13$ non-equivalent $1$-equiparting matrices that are of size $3\times (2\cdot 2^3)$.
\end{proposition}
\begin{proof}
Let $A=(A_1,A_2)$ be a $1$-equiparting matrix.
This means that both $A_1$ and $A_2$ are $3$-bit Gray codes and the last column of $A_1$ is equal to the first column of $A_2$. 
In addition, the transition counts cannot exceed~$5$ and must sum up to~$14$. 
Having in mind that $A$ is a $1$-equiparting matrix it follows that $A$ must have transition counts $\{5,5,4\}$. Hence two of its rows must have transition count $5$ and one row must have transition count $4$. In the following a \emph{realization} of transition counts is a Gray code with the prescribed transition counts.

Since we are counting $1$-equiparting matrices up to equivalence we may fix the first column of $A$, and hence first column of $A_1$, to be $(0,0,0)^t$ and choose for $A_1$ one of the matrices from each of the $3$ classes of $3$-bit Gray codes described in Lemma~\ref{lem:3BitGrayCodes}(\ref{lem:3BitGrayCodes_ii}). 

If $A_1$ has transition counts $(3,2,2)$, i.e., the first row has transition count $3$ while remaining rows have transition count $2$, then its last column is $(1,0,0)^t$. 
The next Gray code $A_2$ in the matrix $a$ can have transition counts $(2,3,2)$, $(2,2,3)$, or $(1,3,3)$, each having $2$ realizations~$A_2$, each with first column $(1,0,0)^t$. 

If $A_1$ has transition $(3,3,1)$, then its last column is $(1,1,0)^t$. The Gray code $A_2$ can have transition counts $(2,2,3)$, having $2$~realizations, or $(1,2,4)$, having $1$~realization, or $(2,1,4)$, having $1$~realization, each with first column $(1,1,0)^t$. 

If $A_1$ has transition counts $(4,2,1)$, then its last column is $(0,0,1)^t$. The Gray code $A_2$ can have transition counts $(1,2,4)$, having $1$~realization, or $(1,3,3)$, having $2$ realizations, each with first column $(0,0,1)^t$.

In total we have $6 + 4 + 3 = 13$  non-equivalent $1$-equiparting matrices $A=(A_1,A_2)$.
\end{proof}

\begin{enumeration}\label{enum:4_3_enumeration}
There are $2015$ non-equivalent $2$-equiparting matrices that are of size $3\times 4\cdot 2^3$.
\end{enumeration}
\begin{proof}
Using Lemma~\ref{lem:3BitGrayCodes} we enumerate non-equivalent $2$-equiparting matrices by computer.
Let $A=(A_1,A_2,A_3,A_4)$ be a  $2$-equiparting matrix.
It must have transition counts $\{10,10,8\}$. 
Similarly as above, $A$ is constructed by fixing the first column to be $(0,0,0)^t$ and $A_1$ to be one representative from each of the $3$ classes of Gray codes. 
Then all possible Gray codes for $A_2,A_3,A_4$ are checked, making sure that the last column of $A_i$ is equal to the first column of $A_{i+1}$ and that the transition counts of $A_1,\ldots,A_4$ sum up to $\{10,10,8\}$. 
This leads to $2015$~possibilities.
\end{proof}

This concludes the proof of Theorem~\ref{th-main:D(j,3)}.

\begin{remark}
By means of a computer we were able to calculate the number $N(j,k,d)$ of non-equivalent $\ell$-equiparting matrices for several values of $j \geq 1$ and $k \geq 3$, where $d=\lceil \tfrac{2^k-1}{k}j\rceil$ and  $\ell = dk - (2^k -1)j$. See Table~\ref{table_num_equip_matrices}.

\begin{table}[ht]
\begin{tabular}{| p{2.1cm} | p{2.1cm} | p{2.1cm} | p{2.1cm} |r| }
\hline
\multicolumn{5}{|c|}{ Number $N(j,k,d)$ of non-equiv $\ell$-equiparting matrices given $j\geq 2$, and $k\geq3$.}\\ \hline \hline
\centering $j$ & \centering $k$ &  \centering $\ell$ &\centering $d$ & $N(j,k,d)$ \\ \hline \hline
\centering 2 & \centering 3 & \centering 1 & \centering 5 & 13 \\ \hline
\centering 3 & \centering 3 & \centering 0 & \centering 7 & 60 \\ \hline
\centering 4 & \centering 3 & \centering 2 & \centering 10 &  2015\\ \hline
\centering 5 & \centering 3 & \centering 1 & \centering 12 & 35040 \\ \hline
\centering 6 & \centering 3 & \centering 0 & \centering 14 & 185130 \\ \hline
\centering 7 & \centering 3 & \centering 2 & \centering 17 & 7572908 \\ \hline
\centering 8 & \centering 3 & \centering 1 & \centering 19 & 132909840 \\ \hline
\centering 9 & \centering 3 & \centering 0 & \centering 21 & 732952248 \\ \hline
\centering 1 & \centering 4 & \centering 1 & \centering 4 & 16 \\ \hline
\centering 2 & \centering 4 & \centering 2 & \centering 8 & 37964 \\ \hline
\end{tabular}
\vspace{5pt}
\caption{\small Number $N(j,k,d)$ of non-equivalent $\ell$-equiparting matrices given $j\geq 2$ and $k\geq 3$, where $d=\lceil \tfrac{2^k-1}{k}j\rceil$ and $\ell = dk - (2^k-1)j$.}\label{table_num_equip_matrices}
\end{table}
\end{remark}

\small
\providecommand{\bysame}{\leavevmode\hbox to3em{\hrulefill}\thinspace} 
\providecommand{\href}[2]{#2}

\end{document}